\documentclass{article}
\usepackage[utf8]{inputenc}
\usepackage[margin=1in]{geometry}
\usepackage[T1]{fontenc}

\usepackage{authblk}
\usepackage{blindtext}

\usepackage{multirow}

\usepackage{a4wide}
\usepackage{epstopdf}
\usepackage{epsfig}
\usepackage{curves}
\usepackage{epic,eepic}
\usepackage{graphicx}
\usepackage{amsmath,amsfonts,amssymb,latexsym}
\usepackage{mathabx}
\usepackage{enumerate}
\usepackage{enumitem}
\usepackage{multirow}
\usepackage{rotating}
\usepackage[thmmarks,amsmath]{ntheorem}
\usepackage{lscape}
\usepackage{float}
\floatstyle{boxed}
\usepackage{authblk}
\usepackage[linesnumbered,ruled,lined,boxed]{algorithm2e}
\usepackage{graphics}
\usepackage{epstopdf}

\usepackage[dvipsnames]{xcolor}

\usepackage{titlesec}
\usepackage{soul}

\usepackage{hyperref}

\newtheorem{proposition}{Proposition}[section]
\newtheorem{theorem}[proposition]{Theorem}

\newtheorem{lemma}[proposition]{Lemma}
\newtheorem{definition}[proposition]{Definition}

\newtheorem{remark}[proposition]{Remark}

\theorembodyfont{\normalfont}

\newenvironment{proof}{\noindent\textbf{Proof.}
  }{\hspace*{\fill}$\Box$ \\[1em]}

\newenvironment{proof-app}[1]{\noindent\textbf{Proof of Proposition \ref{#1}.}
  }{\hspace*{\fill}$\Box$ \\[1em]}

\newcommand{\R}{{\mathbb R}}
\renewcommand{\Re}{{\mathbb R}}

\begin{document}

\title{Scalarization via utility functions \\ in multi-objective optimization}

\author[1]{Lorenzo Lampariello\thanks{lorenzo.lampariello@uniroma3.it}}
\author[2]{Simone Sagratella\thanks{sagratella@diag.uniroma1.it}}
\author[2]{Valerio Giuseppe Sasso\thanks{sasso@diag.uniroma1.it}}
\author[3]{Vladimir Shikhman\thanks{vladimir.shikhman@mathematik.tu-chemnitz.de}}
\affil[1]{Roma Tre University, Italy}
\affil[2]{Sapienza University of Rome, Italy}
\affil[3]{Chemnitz University of Technology, Germany}

\date{\today}

\maketitle

\begin{abstract}
We study a general scalarization approach via utility functions in multi-objective optimization. 
It consists of maximizing utility which is obtained from the objectives’ bargaining with regard to a disagreement reference point. 
The theoretical framework for a broad class of utility functions from microeconomics is developed. 
For that, we associate a utility-dependent single-objective optimization problem with the given multi-objective optimization problem.
We show that Pareto optimal points of the latter can be recovered by solving the former. 
In particular, Cobb-Douglas, Leontief, and CES utility functions are considered.
We prove that any Pareto optimal point can be obtained as a solution of scalarization via one of the mentioned utility functions. 
Further, we propose a numerical
scheme to solve utility-dependent single-objective optimization problems. Here, the main difficulty comes from the necessity to address constraints which are associated with a disagreement reference point.
Our crucial observation is that the explicit treatment of these additional constraints may be avoided. 
This is the case if the Slater condition is satisfied and the utility function under consideration has the so-called
barrier property. 
Under these assumptions, we prove the convergence of our scheme to Pareto optimal points. 
Numerical experiments on real-world financial datasets in a portfolio selection context confirm the efficiency of our scalarization approach via utility functions.
    \end{abstract}
    
\section{Introduction}

In multi-objective optimization, the scalarization approach has meanwhile become standard for trying to find (weak) Pareto optimal points. Usually, the {\it weighted} sum scalarization is applied to tackle a multi-objective optimization problem. This is to optimize a linear combination of the objectives with properly adjusted nonnegative or positive weights. More sophisticated scalarization techniques have been introduced in analogy. We mention the weighted $t$-th power and the weighted quadratic approaches. The whole branch of scalarization techniques successfully incorporates {\it constraints} into the single-objective optimization problem to solve. These are the elastic constraint method and the Benson's method to name a few. The family of {\it reference point} approaches includes a variety of methods in which a properly chosen reference point in the space of objectives is used. Most prominent of them is the distance function approach. Here, the distance from a utopia or ideal point to the Pareto set is minimized. If choosing the maximum norm for the distance, we arrive at the Chebyshev scalarization, and the choice of the $p$-norm leads to the compromise programming. The achievement function method generalizes this idea by allowing to incorporate penalty terms into the single-objective formulation. Another possibility to deal with reference points is via the goal programming. Here, the weighted sum or maximum of negative and positive deviations from the goal to be achieved is minimized. 
Finally, a wide variety of approaches employ a {\it direction} in the objective space
along which a search is performed. Among those we mention the reference direction method, the Pascoletti and Serafini method, and the gauge-based methods. For the scalarization techniques discussed above we refer to the survey \cite{ehrgott2005chapter17} and to the references therein, as well as to the monograph \cite{ehrgott2005multicriteria}.

In this paper, we propose a new scalarization approach for multi-objective optimization via utility functions. The idea is to maximize utility which is obtained from the objectives' bargaining with regard to a prescribed disagreement level. Let us explain our approach in terms of a bargaining game, cf. e.g. \cite{osborne1994course}, where different objectives correspond to the players' losses. This is to say that the players aim to minimize their objectives.
First, in the space of objectives a disagreement point is fixed. It can be viewed as consisting of those losses the players suffer if not coming to an agreement at all. The corresponding bargaining benefits are then given by the differences of the disagreement losses with the players' objectives corresponding to a current agreement. In other words, the benefits indicate to which extent the disagreement losses for every player can be diminished by means of bargaining. Second, these bargaining benefits are combined together via a utility function and the latter is subsequently maximized. The choice of a utility function accounts for the fact that some players and, hence, their bargaining benefits, could be more or less important depending on how they contribute to the social welfare. We point out that the use of utility functions for scalarizing the players' bargaining benefits is original. It allows to link multi-objective optimization to the microeconomics, where utility functions are typically used for measuring the utility of consumption bundles, see e.g. \cite{mascollel1995micro, silberberg2000structure}. 
Since the utility functions are thus usually defined on the nonnegative orthant, we need to take care for the disagreement reference point constraints in the space of objectives.
In economic terms, these constraints prevent the players' bargaining losses to exceed those caused by the disagreement. In order to deal with the disagreement reference point constraints in a mathematically correct way, we impose the Slater condition as it is standard in the convex setting. The use of Slater condition is further justified from the multi-objective perspective.% If violated, the task of finding weak Pareto optimal points just reduces to a feasibility problem, see Proposition \ref{prop: Slat Weak}.

Let us briefly overview our results on the scalarization approach via utility functions. In Section \ref{sec:utilframework}, the theoretical framework for a fairy general class of utility functions is developed.
Here, we associate a utility-dependent single-objective optimization problem with the given multi-objective optimization problem. 
We mainly show that (weak) Pareto optimal points of the latter can be recovered by solving the former, see Theorem \ref{th:strict}. 
Additionally, we state conditions on the utility function under which the corresponding single-objective optimization problem enjoys  efficient solvability. In Section \ref{sec:examples}, we specify our scalarization approach for classical utility functions from microeconomics. 
Namely, Cobb-Douglas, Leontief, and CES utility functions are considered. 
The theoretical framework from Section \ref{sec:utilframework} is thus successively applied for these cases. 
Moreover, we show that any (weak) Pareto optimal point can be obtained as a solution of scalarization via one of the mentioned utility functions, see Theorems \ref{thm:w-cd}, \ref{thm:w-min}, \ref{thm:w-ces-neg}, and \ref{thm:w-ces-pos}. 
For that, the parameters of Cobb-Douglas, Leontief, or CES utility functions need to be properly adjusted. 
Section \ref{sec:algo} is devoted to the algorithmic developments. 
We propose a rather general numerical scheme to solve utility-dependent single-objective optimization problems and thus to obtain Pareto optimal points of the original multi-objective problem.
Here, we focus on the differentiable case for both the objectives and the utility function. 
The main idea is that sometimes the explicit treatment of the difficult to address disagreement reference point constraints may be avoided, see Algorithm \ref{alg:one}. This is at the price of the computation of a starting point satisfying the Slater condition. 
Moreover, the utility function under consideration has to additionally possess the so-called barrier property. 
Under these assumptions, we prove the convergence of our scheme towards a solution of the utility-dependent single-objective optimization problem, see Theorem \ref{th:convergence}, and thus to Pareto solutions.
We also devise Algorithm \ref{alg:two} to compute a Slater point that is instrumental to initialize Algorithm \ref{alg:one}, see Theorem \ref{th:interior starting}. 
Note that the combination of Algorithms \ref{alg:one} and \ref{alg:two} allows one to compute Pareto optimal points even in the case where the Slater condition does not hold, see the discussion below Theorem \ref{th:interior starting}. In Section \ref{Sec:numerical}, we consider a portfolio-selection problem \`a la Markowitz, where an additional Environmental Sustainability and Governance-related objective is taken into account, see \cite{lampariello2023solving}, for which we conduct extensive numerical tests on real-world financial datasets.

Finally, we comment on how our scalarization approach via utility functions fits into the existing literature.
The scalarization by means of the Cobb-Douglas utility function coincides with the weighted geometric mean approach introduced in \cite{lootsma1995controlling}.
Note that this approach is often criticized, see e.g. \cite{audet2008multiobjective}, since it introduces additional disagreement reference point constraints. However, our analysis overcomes this obstacle. In fact, the Cobb-Douglas utility function is shown to be a barrier. Hence, under the Slater condition, the disagreement reference point constraints can be neglected, both from the theoretical and practical perspective -- we obtain only the interior solutions.
This fact has been already observed in \cite{cesarone2018risk}, where the case of two objectives with the special choice of parameters has been treated in the context of mathematical finance. Another link we establish is between the compromise programming, see e.g. \cite{zeleny1973compromise}, and our scalarization by means of the CES utility function. Both approaches can be viewed as in some sense dual to each other, see Remark \ref{rem:compr}. Our approach enlarges the parameter space and allows to go beyond the the case of the $p$-norm in compromise programming. Regardless of the mentioned relations, our scalarization approach via utility functions is novel in its generality and falls into the scope of reference point methods. The presented theoretical results are self-contained and open a door for using other utility function -- except of those already studied here -- for scalarization purposes in multi-objective optimization. 

The main contributions of our paper are summarized below.
\begin{itemize}
    \item We study scalarization of multi-objective problems via utility functions with a degree of generality allowing one to consider, for the first time in the literature, some fundamental functions from microeconomics such as CES.
    \item We devise Algorithm \ref{alg:one} that is shown to be provably convergent to Pareto solutions under the conditions identifying the general framework we focus on. 
    We remark that this procedure needs to consider the disagreement reference point constraints only once at the starting iteration, and it can actually disregard them during the ensuing algorithm progress.
    This directly translates to clear numerical advantages.
    \item We develop Algorithm \ref{alg:two} to compute Slater points for the disagreement reference point constraints, that are essential for initializing Algorithm \ref{alg:one}.
    \item We test our numerical procedures on real-world data in a portfolio-selection context where, apart from the classical Risk-Return related objectives, we include a sustainability-oriented criterion. 
    The numerical results confirm all the theoretical insights, as well as the effectiveness of the solution procedures.
\end{itemize}

Our notation is standard. We denote by $\R^m$ the space of $m$-dimensional vectors, by $\R^m_+$ the set of all vectors with nonnegative components, and by $\R^m_{++}$ the set of all vectors with positive components. 
%If the components of $x \in \R^n$ are nonnegative (positive), we write $x \geq 0$ ($x > 0$). Analogously, 
%For $x, y \in \R^m$ we write $x \geq y$ ($x > y$) if it holds $x_i \geq y_i$ ($x_i > y_i$) for all $i =1,\ldots,n$.
The domain of a function $f:\R^n \rightarrow \R\cup\{-\infty\}$ is $\text{dom} f = \{ x \in \R^n: f(x) > -\infty\}$.
Given a continuously differentiable function $f:\R^n \rightarrow \R\cup\{-\infty\}$, $\nabla f(x)$ denotes its gradient at $x \in \text{dom} f$.

\section{Scalarization via utility functions}\label{sec:utilframework}
We address the general multi-objective optimization problem:
\begin{equation}\label{eq:mobj}
\begin{array}{cl}
\underset{x}{\mbox{minimize}} & f(x) \triangleq \big(f_1(x), \ldots, f_m(x)\big)^\mathsf{T}\\
\mbox{s.t.} & x \in K,
\end{array}
\end{equation}
where $K \subseteq \mathbb R^n$ is the feasible set and $f: \mathbb R^n \to \mathbb R^m$ is the  vector of objective functions. We assume that $K$ is closed and convex, and $f_j$, $j=1, \ldots,m$, are convex. As usual in multi-objective optimization, we focus on the classical notions of (weak) Pareto optimality, see e.g. \cite{ehrgott2005multicriteria}.
 
\begin{definition}
A point $\widehat x \in K$ is called Pareto optimal for \eqref{eq:mobj} if $\forall x \in K$:
$$
  \exists j_x \in \{1,\ldots,m\} \, : \, f_{j_x}(\widehat x) < f_{j_x}(x), \, \text{or} \; \;\forall j \in \{1,\ldots,m\}: \; f_{j}(\widehat x) \leq f_{j}(x).
$$
The set of Pareto optimal points of \eqref{eq:mobj} is denoted by $P$.
\end{definition}

\begin{definition}
A point $\widehat x \in K$ is called weak Pareto optimal for \eqref{eq:mobj} if $\forall x \in K:$
$$
\exists j_x \in \{1,\ldots,m\} \, : \, f_{j_x}(\widehat x) \leq f_{j_x}(x).
$$
The set of weak Pareto optimal points of \eqref{eq:mobj} is denoted by $W$.
\end{definition}

Further, let us fix real numbers $a_j \in \R$ for each of the objective function $f_j$, $j =1, \ldots, m$. We are interested in finding (weak) Pareto optimal points $x \in K$, additionally satisfying the disagreement reference point constraints
\[
f_j(x) \leq a_j, j=1, \ldots,m.
\]
In particular, $a_j$ can be set equal to $f_j(\overline x)$ for some $\overline x \in K$, see e.g. \cite{audet2008multiobjective,cesarone2018risk}, and could be interpreted as a reference level for $f_j$. The vector $a\triangleq(a_1,\ldots,a_m)^\mathsf{T}$ can be then thought of as a disagreement point from the bargaining theory, see e.g. \cite{osborne1994course}. Indeed, if the objective functions of the players $1,\ldots, m$ cannot be successively compromised with each other, $\overline x$ is implemented and the payments $a_1, \ldots, a_m$ are realized. In what follows, we assume that the set 
\begin{equation}
    \label{eq:ineq}
   \left\{ x \in K : f_j(x) \leq a_j, j=1,\ldots,m\right\}
\end{equation}
 is nonempty and bounded. The latter is a technical assumption, which is in particular satisfied for any choice of $a$ if every objective function $f_j$ is bounded below and proper on $K$, $j=1,\ldots,m$. Recall that by definition $f_j$ is proper on $K$ if $f_j^{-1}(C)\cap K$ is compact for any compact set $C \subset \R$. Alternatively, just the compactness of $K$ suffices, in order to guarantee that the set in \eqref{eq:ineq} is bounded.
We eventually use a suitable constraint qualification for (\ref{eq:ineq}).

\begin{definition}
The inequality system in (\ref{eq:ineq}) is said to satisfy the Slater condition if there exists $\overline x \in K$ with 
$$
    f_j(\overline x) < a_j, j=1, \ldots,m.
$$
In this case, we call $\overline x$ a Slater point for \eqref{eq:ineq}.
\end{definition}

 Slater condition is a natural assumption, because in case of its violation the search for weak Pareto optimal points becomes a trivial task. 
 
\begin{proposition} \label{prop: Slat Weak}
    If the Slater condition does not hold for (\ref{eq:ineq}), then all its feasible points are weak Pareto optimal. 
\end{proposition}
\begin{proof}
    We show the contrapositive: assume $x \in K$ exists such that it is feasible for (\ref{eq:ineq}), but $x \notin W$. Then, there exists $\widetilde x \in K$ such that $f_j(\widetilde x) < f_j(x)$ for every $j=1, \ldots, m$. In view of feasibility of $x$, we also have $f_j(x) \leq a_j$ for every $j=1, \ldots, m$. Altogether, $\widetilde x$ is a Slater point for (\ref{eq:ineq}).
\end{proof}
Proposition \ref{prop: Slat Weak} suggests how the reference level $a$ can be chosen, so that Slater condition is automatically satisfied for (\ref{eq:ineq}). By setting $a_j=f_j(x)$, $j=1.\ldots,m$, with some $x \in K \backslash W$, Slater condition holds for (\ref{eq:ineq}). Otherwise, $x$ -- being feasible for (\ref{eq:ineq}) -- would be weak Pareto optimal due to Proposition \ref{prop: Slat Weak}, a contradiction. Another possibility to properly choose the reference level is to use the nadir $N$ defined as
\[
   N_j \triangleq  \max_{x \in P} f_j(x), j=1, \ldots, m.
\]
By taking $a_j > N_j$, $j=1.\ldots,m$, any Pareto optimal point $\overline x \in P$ is a Slater point for (\ref{eq:ineq}). 

Slater condition is also a rather mild assumption. Aiming to give reasons for this, let us assume for a moment that the functions $f_j$, $j=1,\ldots,m$, are continuously differentiable. Then, it is well-known from the theory of nonlinear programming that the linear independence constraint qualification (LICQ) holds on the solution set of a system of generic inequality constraints, see \cite{jongen2000nonlinear}. Applied to (\ref{eq:ineq}), LICQ at its feasible point $ x \in K$ means that the gradients of the active inequality constraints are linearly independent:
\[
   \nabla f_j(x),  j \in J_0(x), 
\]
where
\[
  J_0(x) \triangleq \left\{ j \in \{1,\ldots,m\} : f_j(x)=a_j\right\}.
\]
Genericity here refers to the fact that the set of continuously differentiable defining functions $f_j$, $j=1,\ldots,m$, for which LICQ is fulfilled at every feasible point of (\ref{eq:ineq}), is open and dense with respect to the Whitney topology. It remains to note that LICQ implies Slater condition, see e.g. \cite{stein2018grund}. Thus, the latter turns out to hold for (\ref{eq:ineq}) in the generic sense. 
%Based on the justifications above, we assume in what follows that the Slater condition holds for (\ref{eq:ineq}). 

We mention that the Slater condition imposes a restriction on the feasible weak Pareto optimal points. Namely, they cannot fulfil all the constraints in \eqref{eq:ineq} with equality. This observation is crucial if trying to later identify weak Pareto optimal points as solutions of particularly scalarized optimization problems. 

\begin{proposition}
    \label{prop:slater-rest}
    Let the inequality system \eqref{eq:ineq} satisfy the Slater condition. If $\widehat x \in W$ is feasible for \eqref{eq:ineq}, then there exists $j_{\widehat x} \in \{1,\ldots, m\}$ such that $f_{j_{\widehat x}}(\widehat x) < a_{j_{\widehat x}}$.
\end{proposition}
\begin{proof}
    Assume on the contrary that, due to the feasibility of $\widehat x$ for \eqref{eq:ineq}, we have $f_j(\widehat x) = a_j$ for all $j \in \{1,\ldots,m\}$. But, for a Slater point $\overline{x} \in K$, it also holds $f_j(\overline{x}) < a_j$ for all $j \in \{1,\ldots,m\}$. Altogether, $f_j(\widehat x) > f_j(\overline{x})$ for all $j \in \{1,\ldots,m\}$ contradicts $\widehat x \in W$.
\end{proof}

Now, we are ready to introduce the new scalarization approach for multi-objective optimization. For that, let us consider utility functions $u: \, \mathbb R^m \to \mathbb R \cup \{-\infty\}$. They are typically used in microeconomics for measuring the utility of consumption bundles, i.e. vectors with nonnegative components, see e.g. \cite{silberberg2000structure}. In view of the latter, we assume that the domain of $u$ is the nonnegative orthant, i.e. 
\[
   \text{dom }u = \mathbb R^m_{+}.
\]
We also assume that $u$ is continuous on its domain.
Our scalarization idea consists of solving the following single-objective optimization problem related to \eqref{eq:mobj} and \eqref{eq:ineq}:  
\begin{equation}\label{eq:sobjutil1}
\begin{array}{cl}
\underset{x}{\mbox{maximize}} &  h(x) \triangleq u(a_1- f_1(x), \ldots, a_m- f_m(x))\\
\mbox{s.t.} & x \in K.\\
\end{array}
\end{equation} 

First, we note that solving (\ref{eq:sobjutil1}) leads to the feasible points of (\ref{eq:ineq}).

\begin{proposition}
    \label{prop:feas}
Solutions of \eqref{eq:sobjutil1} are feasible points for \eqref{eq:ineq}.
\end{proposition}
\begin{proof}
    Let $x^* \in K$ solve \eqref{eq:sobjutil1}. For a feasible point $\widehat x \in K$ for \eqref{eq:ineq}, which exists by assumption, it holds $a-f(\widehat x) \in \text{dom }u$. Hence, the optimality of $x^*$ provides $a-f(x^*) \in \text{dom }u$ as well. Since $\text{dom }u = \R^m_+$, the assertion follows.
\end{proof}

In order to enforce that feasible points of \eqref{eq:sobjutil1} are Slater, we need an additional assumption on $u$, which is reminiscent of the condition in \cite[Definition 3.5.1]{miettinen2012nonlinear}. Namely, $u$ has to act as a barrier. This is to say that $u$ is constant on the boundary of its domain and grows towards the interior. 

\begin{definition}
    \label{def:barrier}
We say that the utility function $u$ is a barrier if for some $\overline u \in \mathbb R$ it holds:
\begin{itemize}
\item[(i)] $u(y) = \overline u, \quad \forall y \in \text{bd dom }u=\R^m_+ \backslash \R^m_{++}$;
\item[(ii)] $u(y) > \overline u, \quad \forall y \in \text{int dom }u=\R^m_{++}$.
\end{itemize}
\end{definition}

\begin{proposition}
\label{prop:barrier}
    Let the inequality system \eqref{eq:ineq} satisfy the Slater condition, and the utility function $u$ be a barrier. Then, solutions of \eqref{eq:sobjutil1} are Slater points for \eqref{eq:ineq}.
\end{proposition}
\begin{proof}
    Let $\overline{x} \in K$ be a Slater point for \eqref{eq:ineq}. Since then $a-f(\overline{x}) \in \text{int dom }u$, we obtain from Definition \ref{def:barrier} that $h(\overline{x})=u(a-f(\overline{x}))> \overline{u}$. Let $x^*\in K$ solve \eqref{eq:sobjutil1}, i.e. $h(x^*) \geq h(\overline{x})$. Altogether, we deduce that $h(x^*) > \overline{u}$. Again, Definition \ref{def:barrier} provides that $a-f(x^*) \in \text{int dom }u$. By recalling $\text{int dom }u=\R^m_{++}$, the assertion follows.
\end{proof}

In order to relate the solutions of \eqref{eq:sobjutil1} with (weak) Pareto optimal points of \eqref{eq:mobj}, we recall some standard monotonicity properties for utility functions.

\begin{definition} 
Given a set $S$, the utility function $u$ is called
\begin{itemize}
    \item[(i)] monotone on $S$ if 
    \[
    u(y+\delta) \ge u(y) , \quad \forall y \in S, \, \delta \in \Re^m_+: \,y+\delta \in S;
    \]
    \item[(ii)] weakly strictly monotone on $S$ if
    \[
    u(y+\delta) > u(y) \quad \forall y \in S, \, \delta \in \Re^m_{++}: \,y+\delta \in S;
    \]
    \item[(iii)] strictly monotone on $S$ if
    \[
    u(y+\delta) > u(y) \quad \forall y \in S, \, \delta \in \Re^m_+ \setminus \{0\}: \, y+\delta \in S.
    \]
\end{itemize}
%    If $S=\text{dom } u$, we say that $u$ is monotone, weakly strictly monotone, or strictly monotone.
\end{definition}
We observe that strict monotonicity implies weakly strict monotonicity, which in turn yields monotonicity of $u$. Thanks to next Theorem \ref{th:strict}, one is allowed to obtain (weak) Pareto optima of the original multi-objective problem \eqref{eq:mobj} through the single-objective problem \eqref{eq:sobjutil1}.
\begin{theorem}\label{th:strict}
Let $x^*\in K$ solve \eqref{eq:sobjutil1}.
\begin{itemize}
    \item [(i)] If $u$ is weakly strictly monotone on its domain, then $x^* \in W$.
    \item [(ii)] If $u$ is strictly monotone on its domain, then $x^* \in P$.
    \item[(iii)] If $u$ is strictly monotone on the interior of its domain, and $x^*$ is a Slater point for \eqref{eq:ineq}, then $x^* \in P$.
\end{itemize}
\end{theorem}
\begin{proof}
(i) Assume by contradiction that $x^* \notin W$, then $\widetilde x \in K$ exists such that:
\[
\forall j \in \{1,\ldots,m\} \, : \, f_{j}(\widetilde x) < f_{j}(x^*).
\]
Due to Proposition \ref{prop:feas}, $a-f(x^*) \in \text{dom } u$. Together with $f(\widetilde x) < f(x^*)$, it follows $a-f(\widetilde x) \in \text{dom } u$. By setting $\delta \triangleq f(x^*)-f(\widetilde x)$, we have $\delta \in \R^m_{++}$. The weakly strict monotonicity of $u$ on its domain provides
\[
h(\widetilde x)=u(a-f(\widetilde x))
= u(a-f(x^*)+\delta) > u(a-f(x^*))=
h(x^*),
\]
which contradicts the assumed optimality of $x^*$ for \eqref{eq:sobjutil1}.

(ii) Assume by contradiction that $x^* \notin P$, then $\widetilde x \in K$ exists such that:
\[
\forall j \in \{1,\ldots,m\} \, : \, f_{j}(\widetilde x) \leq f_{j}(x^*), \, \text{and} \; \;\exists j_{\widetilde x} \in \{1,\ldots,m\}: \; f_{j_{\widetilde x}}(\widetilde x) < f_{j_{\widetilde x}}(x^*).
\]
Reasoning similarly to (i), we deduce $a-f(\widetilde x) \in \text{dom} \, u$ from $a-f(x^*) \in \text{dom } u$ and $f(\widetilde x) \leq f(x^*)$. By setting $\delta \triangleq f(x^*)-f(\widetilde x)$, we have now $\delta \in \R^m_+ \backslash\{0\}$.
The strict monotonicity of $u$ on its domain provides
\[
h(\widetilde x)=u(a-f(\widetilde x))
= u(a-f(x^*)+\delta) > u(a-f(x^*))=
h(x^*),
\]
which once again contradicts the assumed optimality of $x^*$ for \eqref{eq:sobjutil1}.

(iii) The proof is the same as (ii) because both $a-f(x^*)$ and $a-f(\widetilde x)$
are in the interior of $\text{dom } u$.
\end{proof}

It is straightforward to see that the (pseudo)concavity of the monotone utility function $u$ is inherited by the objective function $h$. This follows from a general result in convex analysis, see e.g. \cite{avriel2010generalized} and \cite{rockafellar1970convex}. For the sake of completeness, we decided to provide the proof here.  

\begin{proposition} \label{prop: h concave1}
The following statements hold:
\begin{itemize}
    \item[(i)] if $u$ is monotone on its domain and concave, then $h$ is concave;
    \item[(ii)] 
    if $u$ is monotone and pseudoconcave on the interior of its domain, then $h$ is pseudoconcave on the set of the Slater points for \eqref{eq:ineq}, i.e. $\left\{ x \in K : f_j(x) < a_j, j=1,\ldots,m\right\}$.
\end{itemize}
\end{proposition}
\begin{proof}
%Due to the convexity of $K$, we only need to show that $h$ is concave. 
(i) For every $\widehat x, \, \widetilde x \in \mathbb R^n$, and $\lambda \in (0,1)$, the convexity of each $f_j$, $j=1, \ldots,m$, yields
\begin{equation} \label{eq: f convexity2}
f_j ( \lambda \widehat x + (1 - \lambda) \widetilde x) \le  \lambda f_j(\widehat x) + (1 - \lambda) f_j(\widetilde x).
\end{equation}
Thus, we have
\[
\begin{array}{rcl}
h(\lambda \widehat x + (1 - \lambda) \widetilde x) & = & u(a-f(\lambda \widehat x + (1 - \lambda) \widetilde x))
 \ge  u(a-\lambda f(\widehat x) - (1 - \lambda) f(\widetilde x))\\
& = & u(\lambda (a-f(\widehat x)) + (1 - \lambda) (a-f(\widetilde x)))\\
& \ge & \lambda u(a-f(\widehat x)) + (1 - \lambda) u(a-f(\widetilde x))= \lambda h(\widehat x) + (1 - \lambda) h(\widetilde x),
 \end{array}
\]
where the first inequality is due to \eqref{eq: f convexity2} and the monotonicity of $u$ on its domain, and the last one follows from the concavity of $u$. Note that if $a-f(\widehat x) \not \in \text{dom }u$ or $a-f(\widetilde x) \not \in \text{dom }u$, the inequality to be shown is trivially satisfied. Otherwise, $\lambda (a-f(\widehat x)) + (1 - \lambda) (a-f(\widetilde x)) \in \text{dom } u$, due to the convexity of the domain of $u$, and $a-f(\lambda \widehat x + (1 - \lambda) \widetilde x) \in \text{dom }u$, due to \eqref{eq: f convexity2}.
%\[
%  f((1 - \lambda) \widehat x + \lambda \widetilde x) \leq (1 - \lambda)f( \widehat x) + \lambda f(\widetilde x) \leq (1 - \lambda)a + \lambda a =a.
%\]

(ii) Let $\widehat x, \, \widetilde x \in K$ be Slater points for \eqref{eq:ineq} with $h(\widehat x) > h(\widetilde x)$, and $\lambda \in (0,1)$. As above, we have
\[
\begin{array}{rcl}
h((1 - \lambda) \widehat x + \lambda \widetilde x) & = & u(a-f((1 - \lambda) \widehat x + \lambda \widetilde x))
 \ge  u(a-(1 - \lambda)f(\widehat x) - \lambda f(\widetilde x))\\
& = & u((1 - \lambda) (a-f(\widehat x)) + \lambda (a-f(\widetilde x)))\\
& \ge &  u(a-f(\widetilde x)) + \lambda(1 - \lambda) b(\widehat x, \widetilde x) = h(\widetilde x)+\lambda(1 - \lambda) b(\widehat x, \widetilde x)
 \end{array}
\]
with a positive term $b(\widehat x, \widetilde x)$, in general depending on $\widetilde x$ and $\widehat x$.
Here, the first inequality is due to \eqref{eq: f convexity2} and the monotonicity of $u$ on the interior of its domain, and the last one follows from the pseudoconcavity of $u$ by using $u(a-f(\widehat x))> u(a-f(\widetilde x))$. It remains to note that the set of the Slater points for \eqref{eq:ineq} is open and convex.
\end{proof}

Next, we turn our attention to the solvability of the single-objective problem \eqref{eq:sobjutil1}.

\begin{proposition}
    \label{prop:exist}
The solution set of \eqref{eq:sobjutil1} is nonempty.
\end{proposition} 
\begin{proof}
Let us consider the upper level set
\begin{equation}\label{eq:lelvelset}
   {\mathcal{L}}_{\widehat x} \triangleq \{x \in K: h(x) \geq h(\widehat x) \},
\end{equation}
where $\widehat x \in K$ is a feasible point for \eqref{eq:ineq}. Since 
$a-f(\widehat x) \in \text{dom }u$, it holds for every $x \in {\mathcal{L}}_{\widehat x}$ that $a-f(x) \in \text{dom }u$. Recalling $\text{dom }u = \R^m_+$, we obtain
\[
  {\mathcal{L}}_{\widehat x} \subset \left\{ x \in K : f(x) \leq a\right\}.
\]
The assumed boundedness of the latter set provides that ${\mathcal{L}}_{\widehat x}$ is bounded as well. Further, we show that ${\mathcal{L}}_{\widehat x}$ is closed. For that, let an arbitrary sequence $x^k \in {\mathcal{L}}_{\widehat x}$ converge to $\widetilde x$ for $k \rightarrow \infty$. Since $x^k \in K$ and $K$ is closed, we have $\widetilde x\in K$. Moreover,  by arguing as above, $a-f(x^k) \in \text{dom } u$. From the closedness of $\text{dom } u$, we obtain $a-f(\widetilde x) \in \text{dom } u$.
Since it has been assumed that $u$ is continuous on its domain, 
\[
\lim_{k \rightarrow \infty} h(x^k) = \lim_{k \rightarrow \infty} u(a-f(x^k)) = u(a-f(\widetilde x)) =h(\widetilde x).
\]
Hence,  by taking the limit $k \rightarrow \infty$ in $h(x^k) \geq h(\widehat x)$, we get $h(\widetilde x) \geq h(\widehat x)$. Altogether, $\widetilde x \in {\mathcal{L}}_{\widehat x}$ follows.  
Obviously, ${\mathcal{L}}_{\widehat x}$ contains the solution set of \eqref{eq:sobjutil1}.
The existence of a solution of \eqref{eq:sobjutil1} is then due to the Weierstrass theorem by the continuity of $h$ on the nonempty and compact set ${\mathcal{L}}_{\widehat x}$.
\end{proof}

Let us additionally address the unique solvability of the single-objective problem \eqref{eq:sobjutil1}. It turns out that, under the strict concavity of $u$, all solutions of \eqref{eq:sobjutil1} induce the same values of $f_j$'s.

\begin{proposition} \label{prop: h unique1}
    Let $\widehat x \in K$ and $\widetilde x \in K$ both solve \eqref{eq:sobjutil1}.
         \begin{itemize}
         \item[(i)] If $u$ is monotone and strictly concave on its domain, then $f(\widehat x)=f(\widetilde x)$.
        \item[(ii)] If $u$ is monotone and strictly concave on the interior of its domain, and $\widehat x, \widetilde x$ are Slater points for \eqref{eq:ineq}, then $f(\widehat x)=f(\widetilde x)$. 
    \end{itemize}
%   Let $u$ be monotone and strictly concave on its domain. If $\widehat x \in K$ and $\widetilde x \in K$ both solve \eqref{eq:sobjutil1}, then it holds $f(\widehat x)=f(\widetilde x)$.
\end{proposition}
\begin{proof}
   (i) Assume on the contrary that $f(\widehat x) \neq f(\widetilde x)$. From Proposition \ref{prop:feas} we have $a-f(\widehat x) \in \text{dom }u$ and $a-f(\widetilde x) \in \text{dom }u$. Further, we define $x^*\triangleq\frac{1}{2} \widehat x+\frac{1}{2} \widetilde x$ as their average.
   As in the proof of Proposition \ref{prop: h concave1}, it follows $\frac{1}{2}(a-f(\widehat x)) + \frac{1}{2} (a-f(\widetilde x)) \in \text{dom } u$ and $a-f(x^*) \in \text{dom }u$. Analogously, monotonicity and strict concavity of $u$ on its domain imply
   \[
     h(x^*) > \frac{1}{2} h(\widehat x) + \frac{1}{2} h(\widetilde x),
   \]
   where we used that $a-f(\widehat x) \neq a-f(\widetilde x)$. The derived inequality contradicts the fact that for the optimal value of \eqref{eq:sobjutil1} holds $h(\widehat x)=h(\widetilde x)$.

   (ii) The proof is analogous to (ii), by taking into account that $\text{int dom } u=\R^m_{++}$.
\end{proof}

\section{Examples of utility functions}
\label{sec:examples}

For classical utility functions, the conditions identifying the theoretical framework in Section \ref{sec:utilframework} are discussed. To be specific, we elaborate, for suitable choices of the defining parameters, on the monotonicity, concavity, continuity, and barrier properties of these utility functions. Additionally, we show to which extent any (weak) Pareto optimal point can be obtained as a solution of scalarization via the corresponding utility function. We consider Cobb-Douglas, Leontief, and CES utility functions. The latter includes also the linear utility function as a special case.

\subsection{Cobb-Douglas utility function}
First, we consider the Cobb-Douglas utility function, which is widely used for modeling purposes within the economic theory, see e.g. \cite{mascollel1995micro, silberberg2000structure}:
\begin{equation}\label{eq:CobbDoug}
u_{\text{CD}}(y) \triangleq \left\{\begin{array}{cl} \displaystyle
\prod_{j=1}^m y_j^{\alpha_j} & \text{if} \; y \ge 0\\ 
-\infty & \text{otherwise},
\end{array}\right.	
\end{equation}
where the parameters $\alpha_j \geq 0$, $j=1,\ldots,m$, are interpreted as elasticities with respect to the inputs $y_j$'s. 
Basic properties of the Cobb-Douglas utility function are listed in Proposition \ref{prop:list-cd}.

\begin{proposition} \label{prop:list-cd}
 The following statements hold:
    \begin{itemize}
      \item[(i)] $\text{dom }u_{\text{CD}}=\R^m_{+}$; 
      \item[(ii)] $u_{\text{CD}}$ is continuous on its domain;
      \item[(iii)] $u_{\text{CD}}$ is continuously differentiable on the interior of its domain;
      \item[(iv)] if $\alpha_j > 0$ for all $j \in \{1,\ldots,m\}$, then $u_\text{CD}$ is a barrier with $\overline{u}=0$.
    \end{itemize}
\end{proposition}

The monotonicity properties of the Cobb-Douglas utility function are given in Proposition \ref{prop: cobb mono1}.

\begin{proposition} \label{prop: cobb mono1}
    The following statements hold:
    \begin{itemize}
        \item[(i)] $u_{\text{CD}}$ is monotone on its domain;
        \item[(ii)] if $\exists \overline\jmath\in\{1,\ldots,m\}:\alpha_{\overline\jmath} > 0$, then $u_{\text{CD}}$ is weakly strictly monotone on its domain;
        \item[(iii)] if $\alpha_j > 0$ for all $j \in \{1,\ldots,m\}$, then $u_{\text{CD}}$ is strictly monotone on the interior of its domain.
    \end{itemize}
\end{proposition}

\begin{proof}
   (i) For all $y\in \text{dom }u_\text{CD}$ and $\delta \in \mathbb R^m_+$, $y+\delta \in \text{dom }u_{\text{CD}}$, and we have
    \begin{equation*} \label{eq:cobb mon}
    u_{\text{CD}}(y+\delta) = \prod_{j=1}^m (y_j + \delta_j)^{\alpha_j} \geq \prod_{j=1}^m y_j^{\alpha_j} = u_{\text{CD}}(y),
    \end{equation*}
    since $(y_j+\delta_j)^{\alpha_j} \geq  y_j^{\alpha_j} \geq 0$ for all $j \in \{1,\ldots,m\}$. 
    
    (i) Defining $J_{>0}\triangleq \{j \in \{1,\ldots,m \}: \alpha_j >0\}\neq \emptyset$ and $J_{=0}\triangleq \{j \in \{1,\ldots,m \}: \alpha_j =0\}$, for all $y\in \text{dom }u_\text{CD}$ and $\delta \in \mathbb R^m_{++}$, $y+\delta \in \text{int dom }u_\text{CD}$, we have 
    \[
    u_{\text{CD}}(y+\delta) = \prod_{j\in J_{>0}} (y_j + \delta_j)^{\alpha_j} > \prod_{j\in J_{>0}} y_j^{\alpha_j} = u_{\text{CD}}(y),
    \]
    since $(y_j + \delta_j)^{\alpha_j} =  y_j^{\alpha_j}=1$, for all ${j\in J_{=0}}$,
    and $(y_j+\delta_j)^{\alpha_j} > y_j^{\alpha_j} \geq 0$ for all $j\in J_{>0}$.

    (iii) For all $y\in \text{int dom }u_\text{CD}$ and $\delta \in \mathbb R^m_{+}\setminus\{0\}$, $y+\delta \in \text{int dom }u_\text{CD}$, $\exists \overline{\jmath} \in \{1,\ldots,m \}$ such that $\delta_{\overline{\jmath}}>0$, and we have
    \[
    u_{\text{CD}}(y+\delta) = (y_{\overline{\jmath}}+\delta_{\overline{\jmath}})^{\alpha_{\overline{\jmath}}}\prod_{j\neq \overline{\imath}} (y_j + \delta_j)^{\alpha_j} > y_{\overline{\jmath}}^{\alpha_{\overline{\jmath}}}\prod_{j\neq \overline{\jmath}} y_j^{\alpha_j} = u_{\text{CD}}(y),
    \]
    since $(y_{\overline{\jmath}}+\delta_{\overline{\jmath}})^{\alpha_{\overline{\jmath}}} > y_{\overline{\jmath}}^{\alpha_{\overline{\jmath}}}>0$, and $\prod_{j\neq \overline{\jmath}} ( y_j + \delta_j)^{\alpha_j} \geq \prod_{j\neq \overline{\jmath}} y_j^{\alpha_j}>0$.
\end{proof}

In Proposition \ref{prop: cobb conc1} we recall what is known in the economic literature on the concavity properties of $u_\text{CD}$. They vary depending on the sum of the elasticities
    \[
       \alpha\triangleq \sum_{j=1}^m\alpha_j.
    \]
Although the proofs here can be traced back to e.g. \cite[Lemma 5.11 (a), Lemma 5.14 (a)]{avriel2010generalized} and \cite[Theorem 2.1]{avvakumov2010profit}, we decided to show these statements in Appendix for the sake of completeness.

%This assertion can be traced back to \cite[Lemma 5.11 (a)]{avriel2010generalized}. Although the proof can be traced back to \cite[Lemma 5.14 (a)]{avriel2010generalized}, we decided to show this statement for the sake of completeness. For the proof see \cite[Theorem 2.1]{avvakumov2010profit}. 

\begin{proposition}\label{prop: cobb conc1}
    The following statements hold:
    \begin{enumerate}
    \item[(i)]
    if $\alpha \leq 1$, then $u_\text{CD}$ is concave;
    \item[(ii)]
    if $\alpha_j > 0$ for all $j \in \{1,\ldots,m\}$ and $\alpha < 1$, then $u_\text{CD}$ is strictly concave on the interior of its domain;
    \item[(iii)]
    if $\alpha >1$, then $u_\text{CD}$ is pseudoconcave on the interior of its domain.
    \end{enumerate} 
\end{proposition}

Now, we scalarize the multi-objective optimization problem \eqref{eq:mobj} via the Cobb-Douglas utility function. As in \eqref{eq:sobjutil1}, we consider the following single-objective optimization problem:  
\begin{equation}\label{eq:sobjutil-cd}
\begin{array}{cl}
\underset{x}{\mbox{maximize}} &  h_\text{CD}(x) \triangleq u_\text{CD}(a_1- f_1(x), \ldots, a_m- f_m(x))\\
\mbox{s.t.} & x \in K.
\end{array}
\end{equation} 
Explicitly, we thus have
    \begin{equation*}
    h_\text{CD}(x) = \left\{\begin{array}{cl} \displaystyle
    \prod_{j=1}^m (a_j - f_j(x))^{\alpha_j} & \text{if} \; f(x) \leq a\\
    -\infty & \text{otherwise}.
    \end{array}\right.
    \end{equation*}

By applying Propositions \ref{prop:barrier} and \ref{prop:exist} together with Proposition \ref{prop:list-cd} to \eqref{eq:sobjutil-cd}, we obtain the following results regarding its solutions.

\begin{proposition}\label{prop:sol-cd}
    The following statements hold:
   \begin{itemize}
       \item[(i)] the optimization problem \eqref{eq:sobjutil-cd} is solvable;
       \item[(ii)] the solutions of \eqref{eq:sobjutil-cd} are feasible for \eqref{eq:ineq};
        \item[(iii)] if $\alpha_j > 0$ for all $j \in \{1,\ldots,m\}$, then the solutions of \eqref{eq:sobjutil-cd} are Slater points for \eqref{eq:ineq}.
   \end{itemize}
\end{proposition}

Additionally, the unique solvability of \eqref{eq:sobjutil1} can be derived as follows.

\begin{proposition}
 If $\alpha_j > 0$ for all $j \in \{1,\ldots,m\}$, then for the solutions $\widehat x\in K$ and $\widetilde x\in K$ of \eqref{eq:sobjutil-cd} it holds $f(\widehat x)=f(\widetilde x)$.
 \end{proposition}
 
\begin{proof}
   If $\alpha < 1$, we are done by applying Proposition \ref{prop: h unique1} together with Propositions \ref{prop: cobb mono1}, \ref{prop: cobb conc1}, and \ref{prop:sol-cd}.    
   In case $\alpha \geq  1$, let us consider for an arbitrary $\beta > \alpha$ the Cobb-Douglas function 
   \[
      v(y) \triangleq  \left\{\begin{array}{cl} \displaystyle \prod_{j=1}^m y_j^{\gamma_j} & \text{if} \; y \ge 0\\  -\infty & \text{otherwise}, \end{array}\right.	 
   \]   
   where $\gamma_j \triangleq \frac{\alpha_j}{\beta}$, $j=1, \ldots,m$. Since $v(y)=(u_\text{CD}(y))^\frac{1}{\beta}$ and the function $z^{\frac{1}{\beta}}$ with $\beta >1$ is monotonically increasing, $\widehat x$ and $\widetilde x$ remain solutions of \eqref{eq:sobjutil-cd} with $v$ instead of $u_\text{CD}$. In view of $\gamma \triangleq \sum_{j=1}^m\gamma_j < 1$, we may argue as above to obtain $f(\widehat x)= f(\widetilde x)$.
\end{proof}

%From Proposition \ref{prop:exist} we know that . Due to Proposition \ref{prop:feas}, , i.e. $f(x^*) \leq a$. In view of Proposition \ref{prop:barrier}, it can be shown that these inequalities have to strictly hold here.
%
%\begin{proposition}\label{prop: cobb slater}
%    Let $\alpha_j > 0$ for all $j \in \{1,\ldots,m\}$.    If $x^* \in K$ solves \eqref{eq:sobjutil-cd}, then it is a Slater point for \eqref{eq:ineq}, i.e. $f(x^*) < a$.
%\end{proposition}

By applying Theorem \ref{th:strict} together with Propositions \ref{prop: cobb mono1} and \ref{prop:sol-cd}, we obtain (weak) Pareto optima of the original multi-objective problem \eqref{eq:mobj} through the single-objective problem \eqref{eq:sobjutil-cd}.

\begin{proposition}
Let $x^*\in K$ solve \eqref{eq:sobjutil-cd}. The following statements hold: 
        \begin{itemize}
            \item[(i)] if $\exists \overline{\jmath}\in\{1,\ldots,m\}:\alpha_{\overline{\jmath}} > 0$, then $x^* \in W$;
            \item[(ii)] if $\alpha_j > 0$ for all $j\in\{1,\ldots,m\}$, then $x^* \in P$.
        \end{itemize}
\end{proposition}

It turns out that (weak) Pareto optima of \eqref{eq:mobj}, which happen to be feasible or Slater points for \eqref{eq:ineq}, can be found by solving \eqref{eq:sobjutil-cd} if the parameters $\alpha_j$'s are properly adjusted.

\begin{theorem}
\label{thm:w-cd}
  The following statements hold:
  \begin{itemize}
      \item[(i)] if $\widehat{x} \in W$ is feasible for \eqref{eq:ineq} satisfying the Slater condition, then there exist ${\alpha}_j \geq 0$, $j=1,\ldots,m$ with $\exists\overline{\jmath}\in \{1,\ldots,m\}: \alpha_{\overline{\jmath}}>0$, such that $\widehat{x}$ belongs to the solution set of \eqref{eq:sobjutil-cd};
      \item[(ii)] if $\widehat{x} \in P$ is a Slater point for \eqref{eq:ineq}, then there exist ${\alpha}_j > 0$, $j=1,\ldots,m$, such that $\widehat{x}$ belongs to the solution set of \eqref{eq:sobjutil-cd}.
  \end{itemize}
 \end{theorem}
\begin{proof} (i) Preliminarily, consider for $j=1,\ldots,m$:
    \[
    \varphi_j(x)\triangleq \left\{\begin{array}{cl}
-\log(a_j-f_j(x)) & \text{if} \; f_j(x) < a_j\\ 
+\infty & \text{otherwise}.
\end{array}\right.	 
    \] 
    We show that $\widehat x$ is weak Pareto optimal for 
    \begin{equation}\label{eq:mobj phi1}
    \begin{array}{cl}
    \underset{x}{\mbox{minimize}} & \varphi(x) \triangleq \big(\varphi_1(x), \ldots, \varphi_m(x)\big)^\mathsf{T}\\ 
    \mbox{s.t.} & x \in K.
    \end{array}
    \end{equation}
    Otherwise, there would exit $x \in K$, such that for every $j=1,\ldots,m$ it holds 
    \begin{equation}\label{eq:contra-cd}
       \varphi_{j}(\widehat x) > \varphi_{j}(x).    
    \end{equation}
     In particular, it follows $f_j(x) < a_j$.
    In case $f_j(\widehat x) = a_j$, we immediately deduce $f_j(\widehat x) > f_j(x)$.
    If $f_j(\widehat x) < a_j$, \eqref{eq:contra-cd} reads as
\[
- \log(a_{j}-f_{j}(\widehat x)) > -\log(a_{j}-f_{j}(x)).
\]
Hence, again $f_j(\widehat x) > f_j(x)$ follows. Altogether, $\widehat x$ is not weak Pareto optimal for \eqref{eq:mobj}, a contradiction. Further, in view of increasing monotonicity and concavity of the logarithm, reasoning similarly to the proof of Proposition \ref{prop: h concave1}, $\varphi_j$ turns out to be convex for every $j=1,\ldots,m$. By applying \cite[Theorem 4.1]{ehrgott2005multicriteria} to \eqref{eq:mobj phi1}, there exist $\alpha_j \geq 0$, $j=1,\ldots,m$ with $\exists \overline{\jmath}\in\{1,\ldots,m\}:\alpha_{\overline{\jmath}} > 0$, such that $\widehat{x}$ belongs to the solution set of the following problem:
\begin{equation} \label{eq: log cobb1}
    \begin{array}{cl}
    \underset{x}{\mbox{minimize}} & \displaystyle\sum_{j=1}^m \alpha_j \varphi_j(x)\\
    \mbox{s.t.} & x \in K.
    \end{array}
    \end{equation}
Evaluated at a Slater point $\overline{x} \in K$, the objective function in \eqref{eq: log cobb1} is obviously finite. Hence, for every $j \in \{1,\ldots,m\}$ it holds, cf. Proposition \ref{prop:slater-rest},
  \begin{equation}\label{eq:cond-cd}
   \alpha_j \neq 0 \Rightarrow f_j(\widehat x) < a_j.
\end{equation}
Now, we show that $\widehat x$ belongs also to the solution set of \eqref{eq:sobjutil-cd}. In fact, if this were not the case, there would exist $\widetilde x \in K$ such that
\[
  h_\text{CD} (\widetilde x) > h_\text{CD} (\widehat x). 
\]
In particular, $\widetilde x$ must be feasible for \eqref{eq:ineq}. Since $\widehat x$ is feasible for \eqref{eq:ineq} by assumption, the latter inequality becomes
    \begin{equation}\label{eq:in-cd}
    \prod_{j=1}^m (a_j - f_j(\widetilde{x}))^{\alpha_j} > \prod_{j=1}^m (a_j - f_j(\widehat x))^{\alpha_j}.
    \end{equation}
     Because of \eqref{eq:cond-cd} and not all $\alpha_j$'s vanishing, we additionally have
    \[
       \prod_{j=1}^m (a_j - f_j(\widehat x))^{\alpha_j} >0.
    \]
    Hence, we deduce
      \begin{equation}\label{eq:cond-cd1}
   \alpha_j \neq 0 \Rightarrow f_j(\widetilde x) < a_j.
\end{equation}
    Altogether, by applying the logarithm to \eqref{eq:in-cd} and in view of \eqref{eq:cond-cd} and \eqref{eq:cond-cd1}, we obtain
    \[
    -\sum_{j=1}^{m}\alpha_j \log(a_j-f_j(\widetilde{x})) < -\sum_{j=1}^{m}\alpha_j \log(a_j-f_j(\widehat x)),
    \]
    This provides a contradiction to the fact that $\widehat x$ solves \eqref{eq: log cobb1}.
    
    (ii) The proof is analogous to (i).
%If $x^*$ is a solution of \eqref{eq:sobjutil}, due to the Slater condition, $x^* \in K \cap \text{int dom }h$, i.e. $f(x^*)<a$. Moreover, $\prod_{j=1}^m (a_j - f_j(x^*))^{\alpha_j} \geq \prod_{j=1}^m (a_j - f_j(x))^{\alpha_j}$ for all $x \in K \cap \text{int dom }h$. Thus, $\sum_{j=1}^{m}\alpha_j \log(a_j-f_j(x^*))\geq \sum_{j=1}^{m}\alpha_j \log(a_j-f_j(x))$ for all $x \in K \cap \text{int dom }h$, in view of the monotonicity of $\log$ on $\text{int dom }h$ and in turn, $x^*$ solves also \eqref{eq: log cobb}. 
\end{proof}

\subsection{Leontief utility function}

Next, we consider the Leontief utility function, see e.g. \cite{mascollel1995micro}:
%which leads to the robust approach for the multi-objective problem \eqref{eq:mobj}: 
\begin{equation}\label{eq:MIN}
u_\text{MIN}(y) \triangleq \left\{\begin{array}{cl}
\displaystyle \min_{j=1,\ldots,m} \alpha_j \, y_j & \text{if} \; y \ge 0\\
-\infty & \text{otherwise},
\end{array}
\right.
\end{equation}
where the parameters $\alpha_j \geq 0$, $j=1,\ldots,m$, play the role of inputs' weights.
Basic properties of the Leontief utility function are listed in Proposition \ref{prop:list-min}.

\begin{proposition} \label{prop:list-min}
 The following statements hold:
    \begin{itemize}
      \item[(i)] $\text{dom }u_{\text{MIN}}=\R^m_{+}$; 
      \item[(ii)] $u_{\text{MIN}}$ is continuous on its domain; 
      \item[(iii)] if $\alpha_j > 0$ for all $j \in \{1,\ldots,m\}$, then $u_\text{MIN}$ is a barrier with $\overline{u}=0$.
    \end{itemize}
\end{proposition}

The monotonicity properties of the Leontief utility function are given in Proposition \ref{prop: min mono1}.

\begin{proposition} \label{prop: min mono1}
    The following statements hold:
    \begin{itemize}
        \item[(i)] $u_\text{MIN}$ is monotone on its domain;
        \item[(ii)] if $\alpha_j > 0$ for all $j \in \{1,\ldots, m\}$, then $u_\text{MIN}$ is weakly strictly monotone on its domain.
    \end{itemize}
\end{proposition}
{\begin{proof}
    (i) For all $y\in \text{dom }u_\text{MIN}$ and $\delta \in \mathbb R^m_+$, $y+\delta \in \text{dom }u_\text{MIN}$, and we have
    \begin{equation*} \label{eq:leon mon}
    u_\text{MIN}(y+\delta) = \min_{j=1, \ldots, m} \alpha_j(y_j + \delta_j) \geq \min_{j=1, \ldots, m} \alpha_j y_j = u_\text{MIN}(y),
    \end{equation*}
    since $\alpha_j(y_j+\delta_j) \geq \alpha_j y_j$ for all $j \in \{1,\ldots,m\}$. 

    (ii) For all $y\in \text{dom }u_\text{MIN}$ and $\delta \in \mathbb R^m_{++}$, $y+\delta \in \text{dom }u_\text{MIN}$, and we have
    \begin{equation*} 
    u_\text{MIN}(y+\delta) = \min_{j=1, \ldots, m} \alpha_j(y_j + \delta_j) > \min_{j=1, \ldots, m} \alpha_j y_j = u_\text{MIN}(y),
    \end{equation*}
    since $\alpha_j(y_j+\delta_j) > \alpha_j y_j$ for all $j\in \{1,\ldots,m\}$.
\end{proof}}

The concavity of the Leontief utility function is obvious, since it is represented as the minimum of linear functions. 

\begin{proposition}
     $u_\text{MIN}$ is concave.
\end{proposition}
%\begin{proof}
%     The proof follows from \cite[Proposition 2.18]{avriel2010generalized}.
%\end{proof}

Now, we scalarize the multi-objective optimization problem \eqref{eq:mobj} via the Leontief utility function. As in \eqref{eq:sobjutil1}, we consider the following single-objective optimization problem:  
\begin{equation}\label{eq:sobjutil-min}
\begin{array}{cl}
\underset{x}{\mbox{maximize}} &  h_\text{MIN}(x) \triangleq u_\text{MIN}(a_1- f_1(x), \ldots, a_m- f_m(x))\\
\mbox{s.t.} & x \in K.\\
\end{array}
\end{equation} 
Explicitly, we thus have
    \begin{equation*}
    h_\text{MIN}(x) = \left\{\begin{array}{cl} \displaystyle
    \min_{j=1,\ldots,m} \alpha_j(a_j - f_j(x)) & \text{if} \; f(x) \leq a\\
    -\infty & \text{otherwise}.
    \end{array}\right.
    \end{equation*}

By applying Propositions \ref{prop:exist}-\ref{prop:barrier} together with Proposition \ref{prop:list-min} to \eqref{eq:sobjutil-min}, we obtain the following results regarding its solutions.

\begin{proposition}\label{prop:sol-min}
    The following statements hold:
   \begin{itemize}
       \item[(i)] the optimization problem \eqref{eq:sobjutil-min} is solvable;
       \item[(ii)] the solutions of \eqref{eq:sobjutil-min} are feasible for \eqref{eq:ineq};
        \item[(iii)] if $\alpha_j > 0$ for all $j \in \{1,\ldots,m\}$, then the solutions of \eqref{eq:sobjutil-min} are Slater points for \eqref{eq:ineq}.
   \end{itemize}
\end{proposition}

%From Proposition \ref{prop:exist} we know that \eqref{eq:sobjutil-min} is solvable. Due to Proposition \ref{prop:feas}, its solutions $x^* \in K$ are feasible for \eqref{eq:ineq}, i.e. $f(x^*) \leq a$. In view of Proposition \ref{prop:barrier}, these inequalities have to strictly hold here.

%\begin{proposition}\label{prop: min slater}
%    Let $\alpha_j > 0$ for all $j \in \{1,\ldots,m\}$.    If $x^* \in K$ solves \eqref{eq:sobjutil-min}, then it is a Slater point for \eqref{eq:ineq}, i.e. $f(x^*) < a$.
%\end{proposition}

By applying Theorem \ref{th:strict} together with Proposition \ref{prop: min mono1}, we obtain weak Pareto optima of the original multi-objective problem \eqref{eq:mobj} through the single-objective problem \eqref{eq:sobjutil-min}.

\begin{proposition}
Let $x^*\in K$ solve \eqref{eq:sobjutil-min}. If $\alpha_j > 0$ for all $j\in\{1,\ldots,m\}$, then $x^* \in W$.
\end{proposition}

It turns out that (weak) Pareto optima of \eqref{eq:mobj}, which happen to be feasible or Slater points for \eqref{eq:ineq}, can be found by solving \eqref{eq:sobjutil-min} if the parameters $\alpha_j$'s are properly adjusted.

\begin{theorem}
 \label{thm:w-min}
The following statements hold:
 \begin{itemize}
     \item[(i)] if $\widehat{x} \in P$ is feasible for \eqref{eq:ineq} satisfying the Slater condition, then there exist ${\alpha}_j \geq 0$, $j=1,\ldots,m$ with $\exists\overline{\jmath}\in \{1,\ldots,m\}: \alpha_{\overline{\jmath}}>0$, such that $\widehat{x}$ belongs to the solution set of \eqref{eq:sobjutil-min};
     \item[(ii)] if $\widehat{x} \in W$ is a Slater point for \eqref{eq:ineq}, then there exist ${\alpha}_j > 0$, $j=1,\ldots,m$, such that $\widehat{x}$ belongs to the solution set of \eqref{eq:sobjutil-min}.
 \end{itemize}
\end{theorem}

\begin{proof}
(i) By using the index subset
\[
 J_{>}(\widehat x)\triangleq \{j \in \{1,\ldots,m \}: f_j(\widehat x) < a_j\},
\]
we set
    \[
    \alpha_j \triangleq \left\{\begin{array}{cl} \displaystyle
    \frac{1}{a_j-f_j(\widehat x)} & \text{if} \; j \in J_{>}(\widehat x)\\
    0 & \text{otherwise}.
    \end{array}\right.
    \]
Proposition \ref{prop:slater-rest} shows that $J_{>}(\widehat x) \neq \emptyset$. 
%Otherwise, due to the feasibility of $\widehat x$ for \eqref{eq:ineq}, we would have $f_j(\widehat x) = a_j$ for all $j \in \{1,\ldots,m\}$. But, for a Slater point $\overline{x} \in K$, it also holds $f_j(\overline{x}) < a_j$ for all $j \in \{1,\ldots,m\}$. Altogether, $f_j(\widehat x) > f_j(\overline{x})$ for all $j \in \{1,\ldots,m\}$ contradicts $\widehat x \in P$. 
Then, ${\alpha}_j \geq 0$, $j=1,\ldots,m$ with $\exists\overline{\jmath}\in \{1,\ldots,m\}: \alpha_{\overline{\jmath}}>0$. Moreover, we have
\begin{equation} \label{eq:contra-min}
    \min_{j=1, \ldots, m} \alpha_j(a_j - f_j(\widehat x)) = \alpha_j(a_j - f_j(\widehat x)), \quad \forall j \in J_{>}(\widehat x).
\end{equation}
Let us show that $\widehat x$ belongs to the solution set of \eqref{eq:sobjutil-min}.
Assume by contradiction that this is not the case, i.e. there exists $\widetilde x \in K$ with
\[
  h_\text{MIN} (\widetilde x) > h_\text{MIN} (\widehat x). 
\]
In particular, $\widetilde x$ must be feasible for \eqref{eq:ineq}. Since $\widehat x$ is feasible for \eqref{eq:ineq}, the latter inequality becomes
\[
\underset{j=1, \ldots, m}{\min} \alpha_j(a_j - f_j(\widetilde x )) > \underset{j=1, \ldots, m}{\min} \alpha_j(a_j - f_j(\widehat x)).
\]
Due to \eqref{eq:contra-min}, it follows
\[
\alpha_j(a_j - f_j(\widetilde x )) > \alpha_j(a_j - f_j(\widehat{x} )) \quad \forall j \in J_{0}(\widehat x).
\]
Thus, $f_j(\widetilde x)<f_j(\widehat x)$ for every $j \in J_{>}(\widehat x)$.
If $j \not \in J_{>}(\widehat x)$, $f_j(\widehat x) = a_j$ and $f_j(\widetilde x) \leq a_j$, the latter due to the feasibility of $\widetilde x$ for \eqref{eq:ineq}, provide together $f_j(\widetilde x) \leq f_j(\widehat x)$. This is a contradiction to $\widehat x \in P$.

(ii) The proof is analogous to (i).
\end{proof}

%Similarly to the proof of Proposition \ref{prop:w-min}, weak Pareto points of \eqref{eq:mobj}, which happen to be Slater, are solutions of properly parametrized single-objective optimization problems \eqref{eq:sobjutil-min}. 

%\begin{proposition}
% \label{prop:w-min-slater}
% If $\widehat{x} \in W$ is a Slater point for \eqref{eq:ineq}, then there exist ${\alpha}_j > 0$, $j=1,\ldots,m$, such that $\widehat{x}$ belongs to the solution set of \eqref{eq:sobjutil-min}.
%\end{proposition}

%\begin{proof}
%Since $a_j -f_j(\widehat x) >0$, we define $\alpha_j = \frac{1}{a_j-f_j(\widehat x)}$ for every $j \in \{1,\ldots,m\}$. 
%Then, we have
%\[
%\min_{j=1, \ldots, m} \{\alpha_j(a_j - f_j(\widehat x))\} = \alpha_j(a_j - f_j(\widehat x)).
%\]
%We now show that $\widehat x$ belongs to the solution set of \eqref{eq:sobjutil-min}.
%Assume by contradiction that this is not the case, i.e. there exists $\widetilde x \in K$ with
%\[
%\alpha_j(a_j - f_j(\widetilde x )) \geq \underset{j=1, \ldots, m}{\min} \alpha_j(a_j - f_j(\widetilde x )) > \underset{j=1, \ldots, m}{\min} \alpha_j(a_j - f_j(\widehat x)) = \alpha_j(a_j - f_j(\widehat{x} )).
%\]
%Thus, $f_j(\widetilde x)<f_j(\widehat x)$ for every $j \in \{1,\ldots,m\}$, in contradiction to $\widehat x \in W$.
%\end{proof}

\subsection{CES utility function}
\label{subsec:ces}
Another prominent utility function is that of constant elasticity of substitution (CES), see e.g. \cite{mascollel1995micro, silberberg2000structure}: 
%Usually, it is defined just on the positive orthant. Since we are interested in a continuous extension of the CES utility function on the nonnegative orthant, the classical definition has to be modified accordingly: 
%\begin{itemize}
%\item case $\rho \in (0,1]$
    \begin{equation}\label{eq:CES}
u_\text{CES}(y) \triangleq \left\{\begin{array}{cl}
\displaystyle \left(\sum_{j=1}^m \alpha_j \, y_j^{\rho}\right)^{\frac{\kappa}{\rho}} & \text{if} \; y \geq 0\\
-\infty & \text{otherwise},
\end{array}
\right.
\end{equation}
%    \item case $\rho \in (-\infty,0)$
%    \begin{equation}\label{eq:CES}
%u_\text{CES}(y) \triangleq \left\{\begin{array}{cl} \displaystyle \left(\sum_{j=1}^m \alpha_j \, y_j^{\rho}\right)^{\frac{\kappa}{\rho}} & \text{if} \; y > 0\\
%0& \text{if} \; \exists j\in\{1,\ldots,m\}:y_j = 0\\ -\infty & \text{otherwise},
%\end{array}
%\right.
%\end{equation}
%\end{itemize}
where the parameters $\alpha_j \geq 0$, $j=1,\ldots,m$, play the role of inputs' weights, the parameter $\kappa \in (0,1]$ is the degree of homogeneity, and the parameter $\rho \leq 1$ with $\rho\neq 0$ allows to incorporate different substitution patterns. 
Aiming to illustrate the latter, let us consider for a moment the case of $\kappa=1$. %linked to the elasticity of substitution $\frac{1}{1-\rho}$. 
For $\rho=1$, $u_\text{CES}$ is just the linear utility function
\begin{equation}\label{eq:uLIN}
u_\text{LIN}(y) \triangleq \left\{\begin{array}{cl}
\displaystyle \sum_{j=1}^m \alpha_j \, y_j & \text{if} \; y \geq 0\\
-\infty & \text{otherwise}.
\end{array}
\right.
\end{equation}
In $u_\text{LIN}$ the inputs can perfectly substitute each other. If $\rho \rightarrow -\infty$, $u_\text{CES}$ leads to the Leontieff utility function $u_\text{MIN}$.
In $u_\text{MIN}$ the inputs are perfect complements. We also mention that $u_\text{CES}$ becomes the Cobb-Douglas utility function $u_\text{CD}$ if $\rho \rightarrow 0$.
Note that while evaluating $u_\text{CES}$ we use the standard convention $\frac{1}{0}=\infty$ and $\frac{1}{\infty}=0$.
Basic properties of the CES utility function are listed in Proposition \ref{prop:list-ces}.

\begin{proposition} \label{prop:list-ces}
 The following statements hold:
    \begin{itemize}
      \item[(i)] $\text{dom }u_{\text{CES}}=\R^m_{+}$; 
      \item[(ii)] $u_{\text{CES}}$ is continuous on its domain;
      \item[(iii)] $u_{\text{CES}}$ is continuously differentiable on the interior of its domain;
      \item[(iv)] if $\alpha_j > 0$ for all $j \in \{1,\ldots,m\}$ and $\rho <0$, then $u_\text{CES}$ is a barrier with $\overline{u}=0$.
    \end{itemize}
\end{proposition}

The monotonicity properties of the CES utility function are given in Proposition \ref{prop: ces mono1}.

\begin{proposition} \label{prop: ces mono1}
    The following statements hold:
    \begin{itemize}
        \item[(i)] $u_{\text{CES}}$ is monotone on its domain;
        \item[(ii)] if $\exists \overline\jmath\in\{1,\ldots,m\}:\alpha_{\overline\jmath} > 0$, then $u_{\text{CES}}$ is weakly strictly monotone on its domain;
        \item[(iii)] if $\alpha_j > 0$ for all $j \in \{1,\ldots,m\}$, then $u_{\text{CES}}$ is strictly monotone on its domain.
    \end{itemize}
\end{proposition}

\begin{proof}
    (i) For all $y\in \text{dom }u_\text{CES}$ and $\delta \in \mathbb R^m_+$, $y+\delta \in \text{dom }u_{\text{CES}}$, we have for all $j \in \{1,\ldots,m\}$
    \[
    \alpha_j(y_j + \delta_j)^{\rho} \geq \alpha_j y_j^{\rho} \mbox{ if } 0 <\rho \leq 1,
    \mbox{ and }
    \alpha_j(y_j + \delta_j)^{\rho} \leq \alpha_j y_j^{\rho} \mbox{ if } \rho <0.
    \]
    From here we immediately obtain
    \[
    u_\text{CES}(y+\delta)=\left(\sum_{j=1}^{m}\alpha_j(y_j + \delta_j)^{\rho} \right)^{\frac{\kappa}{\rho}} \geq \left(\sum_{j=1}^{m}\alpha_j y_j^{\rho} \right)^{\frac{\kappa}{\rho}} = u_\text{CES}(y),
    \]
    
   (ii) Defining $J_{>0}\triangleq \{j \in \{1,\ldots,m \}: \alpha_j >0\}\neq \emptyset$ and $J_{=0}\triangleq \{j \in \{1,\ldots,m \}: \alpha_j =0\}$, for all $y\in \text{dom }u_\text{CES}$ and $\delta \in \mathbb R^m_{++}$, $y+\delta \in \text{dom }u_\text{CES}$, we have 
    \[
    u_{\text{CES}}(y+\delta) = \left(\sum_{j\in J_{>0}} \alpha_j(y_j + \delta_j)^\rho \right)^{\frac{\kappa}{\rho}}> \left(\sum_{j\in J_{>0}} \alpha_j y_j^\rho \right)^{\frac{\kappa}{\rho}} = u_{\text{CES}}(y).
    \]
    
    (iii) For all $y\in \text{dom }u_{\text{CES}}$ and $\delta \in \mathbb R^m_{+}\setminus\{0\}$, $y+\delta \in \text{dom }u_{\text{CES}}$, $\exists \overline{\jmath} \in \{1,\ldots,m \}$ such that $\delta_{\overline{\jmath}}>0$, and we have
    \[
    \alpha_{\overline{\jmath}}(y_{\overline{\jmath}} + \delta_{\overline{\jmath}})^{\rho} > \alpha_{\overline{\jmath}} y_{\overline{\jmath}}^{\rho} \mbox{ if } 0 <\rho \leq 1,
    \mbox{ and }
    \alpha_{\overline{\jmath}}(y_{\overline{\jmath}} + \delta_{\overline{\jmath}})^{\rho} < \alpha_{\overline{\jmath}} y_{\overline{\jmath}}^{\rho} \mbox{ if } \rho <0,
    \]
    and for all $j \neq \overline{\jmath}$
    \[
    \alpha_j(y_j + \delta_j)^{\rho} \geq \alpha_j y_j^{\rho} \mbox{ if } 0 <\rho \leq 1,
    \mbox{ and }
    \alpha_j(y_j + \delta_j)^{\rho} \leq \alpha_j y_j^{\rho} \mbox{ if } \rho <0.
    \]
    As above, it follows
    \[
     u_{\text{CES}}(y+\delta) = \left( \alpha_{\overline{\jmath}}(y_{\overline{\jmath}} + \delta_{\overline{\jmath}})^\rho +\sum_{j\neq \overline{\jmath}} \alpha_j(y_j + \delta_j)^\rho \right)^{\frac{\kappa}{\rho}}> \left(\alpha_{\overline{\jmath}}y_{\overline{\jmath}}^\rho +\sum_{j\neq \overline{\jmath}} \alpha_j y_j^\rho \right)^{\frac{\kappa}{\rho}} = u_{\text{CES}}(y).
    \] 
\end{proof}

In Proposition \ref{prop: ces conc1} we elaborate on the concavity properties of $u_\text{CES}$. 
Although they are thoroughly studied in the economic literature, see e.g. \cite[Theorem 3.1]{avvakumov2010profit}, we decided to provide the corresponding proofs in Appendix for the sake of completeness.

\begin{proposition}\label{prop: ces conc1}
The following statements hold:
\begin{itemize}
    \item[(i)] $u_\text{CES}$ is concave;
    \item[(ii)] if $\alpha_j > 0$ for all $j \in \{1,\ldots,m\}$, $\kappa \in (0,1)$ and $\rho \in (0,1)$, then  $u_\text{CES}$ is strictly concave on its domain;
    \item[(iii)] if $\alpha_j > 0$ for all $j \in \{1,\ldots,m\}$, $\kappa \in (0,1)$ and $\rho <0$, then  $u_\text{CES}$ is strictly concave on the interior of its domain.
\end{itemize}    
\end{proposition}

Now, we scalarize the multi-objective optimization problem \eqref{eq:mobj} via the CES utility function. As in \eqref{eq:sobjutil1}, we consider the following single-objective optimization problem:  
\begin{equation}\label{eq:sobjutil-ces}
\begin{array}{cl}
\underset{x}{\mbox{maximize}} &  h_\text{CES}(x) \triangleq u_\text{CES}(a_1- f_1(x), \ldots, a_m- f_m(x))\\
\mbox{s.t.} & x \in K.\\
\end{array}
\end{equation} 
Explicitly, we thus have
    \begin{equation*}
    h_\text{CES}(x) = \left\{\begin{array}{cl}
    \left(\displaystyle \sum_{j=1}^m \alpha_j(a_j - f_j(x))^{\rho} \right)^{\frac{\kappa}{\rho}} & \text{if} \; f(x) \leq a\\
    -\infty & \text{otherwise}.
    \end{array}\right.
    \end{equation*}

By applying Propositions \ref{prop:exist}-\ref{prop:barrier} together with Proposition \ref{prop:list-ces} to \eqref{eq:sobjutil-ces}, we obtain the following results regarding its solutions.

\begin{proposition}\label{prop:sol-ces}
    The following statements hold:
   \begin{itemize}
       \item[(i)] the optimization problem \eqref{eq:sobjutil-ces} is solvable;
       \item[(ii)] the solutions of \eqref{eq:sobjutil-ces} are feasible for \eqref{eq:ineq};
        \item[(iii)] if $\alpha_j > 0$ for all $j \in \{1,\ldots,m\}$ and $\rho <0$, then the solutions of \eqref{eq:sobjutil-ces} are Slater points for \eqref{eq:ineq}.
   \end{itemize}
\end{proposition}

%From Proposition \ref{prop:exist} we know that \eqref{eq:sobjutil-ces} is solvable. Due to Proposition \ref{prop:feas}, its solutions $x^* \in K$ are feasible for \eqref{eq:ineq}, i.e. $f(x^*) \leq a$. Moreover, it can be shown that these inequalities have to strictly hold here.

%\begin{proposition}\label{prop: ces slater}
%    Let $\alpha_j > 0$ for all $j \in \{1,\ldots,m\}$.    If $x^* \in K$ solves \eqref{eq:sobjutil-ces}, then it is a Slater point for \eqref{eq:ineq}, i.e. $f(x^*) < a$.
%\end{proposition}
%\begin{proof}
%    We note that $u_{CES}$ is a barrier. Indeed, we can set $\overline{u}=-\infty$ in Definition \ref{def:barrier}. It just remains to apply Proposition \ref{prop:barrier}.
%\end{proof}

Additionally, the unique solvability of \eqref{eq:sobjutil-ces} can be derived by the application of Proposition \ref{prop: h unique1} together with Propositions \ref{prop: ces mono1}, \ref{prop: ces conc1}, and \ref{prop:sol-ces}.

\begin{proposition}
 If $\alpha_j > 0$ for all $j \in \{1,\ldots,m\}$ and $\kappa \in (0,1)$, then for the solutions $\widehat x\in K$ and $\widetilde x\in K$ of \eqref{eq:sobjutil-ces} it holds $f(\widehat x)=f(\widetilde x)$.
 \end{proposition}

By applying Theorem \ref{th:strict} together with Proposition \ref{prop: ces mono1}, we obtain (weak) Pareto optima of the original multi-objective problem \eqref{eq:mobj} through the single-objective problem \eqref{eq:sobjutil-ces}.

\begin{proposition}
Let $x^*\in K$ solve \eqref{eq:sobjutil-ces}. The following statements hold: 
        \begin{itemize}
            \item[(i)] if $\exists \overline{\jmath}\in\{1,\ldots,m\}:\alpha_{\overline{\jmath}} > 0$, then $x^* \in W$;
            \item[(ii)] if $\alpha_j > 0$ for all $j\in\{1,\ldots,m\}$, then $x^* \in P$.
        \end{itemize}
\end{proposition}

It turns out that (weak) Pareto optima of \eqref{eq:mobj}, which happen to be feasible or Slater points for \eqref{eq:ineq}, can be found by solving \eqref{eq:sobjutil-ces} if the parameters $\alpha_j$'s are properly adjusted. Let us treat the cases $\rho <0$ and $\rho \in (0,1]$ separately.

\begin{theorem}
 \label{thm:w-ces-neg}
  The following statements hold for $\rho <0$:
  \begin{itemize}
      \item[(i)] if $\widehat{x} \in W$ is feasible for \eqref{eq:ineq} satisfying the Slater condition, then there exist ${\alpha}_j \geq 0$, $j=1,\ldots,m$ with $\exists\overline{\jmath}\in \{1,\ldots,m\}: \alpha_{\overline{\jmath}}>0$, such that $\widehat{x}$ belongs to the solution set of \eqref{eq:sobjutil-ces};
      \item[(ii)] if $\widehat{x} \in P$ is a Slater point for \eqref{eq:ineq}, then there exist ${\alpha}_j > 0$, $j=1,\ldots,m$, such that $\widehat{x}$ belongs to the solution set of \eqref{eq:sobjutil-ces}.
  \end{itemize}  
\end{theorem}
\begin{proof} (i)
 Preliminary, consider for $j=1,\ldots,m$:
    \[
    \varphi_j(x)\triangleq \left\{\begin{array}{cl}
(a_j-f_j(x))^{\rho} & \text{if} \; f_j(x) < a_j\\
+\infty & \text{otherwise}.
\end{array}\right.	 
    \] 
     We show that $\widehat x$ is weak Pareto optimal for 
    \begin{equation}\label{eq:mobj phi1-ces}
    \begin{array}{cl}
    \underset{x}{\mbox{minimize}} & \varphi(x) \triangleq \big(\varphi_1(x), \ldots, \varphi_m(x)\big)^\mathsf{T}\\
    \mbox{s.t.} & x \in K.
    \end{array}
    \end{equation}
    Otherwise, there would exit $x \in K$, such that for every $j=1,\ldots,m$ it holds 
    \begin{equation}\label{eq:contra-ces}
       \varphi_{j}(\widehat x) > \varphi_{j}(x).    
    \end{equation}
     In particular, it follows $f_j(x) < a_j$.
    In case $f_j(\widehat x) = a_j$, we immediately deduce $f_j(\widehat x) > f_j(x)$.
    If $f_j(\widehat x) < a_j$, \eqref{eq:contra-ces} reads as
\[
(a_j-f_j(\widehat x))^{\rho} > (a_j-f_j(x))^{\rho}.
\]
Hence, again $f_j(\widehat x) > f_j(x)$ follows. Altogether, $\widehat x$ is not weak Pareto optimal for \eqref{eq:mobj}, a contradiction. Further, in view of decreasing monotonicity and convexity of the function $z^{\rho}$ for $\rho <0$, reasoning similarly to the proof of Proposition \ref{prop: h concave1}, $\varphi_j$ turns out to be convex for every $j=1,\ldots,m$. By applying \cite[Theorem 4.1]{ehrgott2005multicriteria} to \eqref{eq:mobj phi1-ces}, there exist $\alpha_j \geq 0$, $j=1,\ldots,m$ with $\exists \overline{\jmath}\in\{1,\ldots,m\}:\alpha_{\overline{\jmath}} > 0$, such that $\widehat{x}$ belongs to the solution set of the following problem:
\begin{equation} \label{eq: log ces1}
    \begin{array}{cl}
    \underset{x}{\mbox{minimize}} & \displaystyle\sum_{j=1}^m \alpha_j \varphi_j(x)\\
    \mbox{s.t.} & x \in K.\\
    \end{array}
    \end{equation}
Evaluated at a Slater point $\overline{x} \in K$, the objective function in \eqref{eq: log ces1} is obviously finite. Hence, for every $j \in \{1,\ldots,m\}$ it holds, cf. Proposition \ref{prop:slater-rest},
  \begin{equation}\label{eq:cond-ces}
   \alpha_j \neq 0 \Rightarrow f_j(\widehat x) < a_j.
\end{equation}
Now, we show that $\widehat x$ belongs also to the solution set of \eqref{eq:sobjutil-cd}. In fact, if this were not the case, there would exist $\widetilde x \in K$ such that
\[
  h_\text{CES} (\widetilde x) > h_\text{CES} (\widehat x). 
\]
In particular, $\widetilde x$ must be feasible for \eqref{eq:ineq}. Since $\widehat x$ is feasible for \eqref{eq:ineq} by assumption, the latter inequality becomes
    \begin{equation}\label{eq:in-ces}
\left(\sum_{j=1}^m \alpha_j(a_j - f_j(\widetilde x))^{\rho} \right)^{\frac{\kappa}{\rho}}
> \left(\sum_{j=1}^m \alpha_j(a_j - f_j(\widehat x))^{\rho} \right)^{\frac{\kappa}{\rho}}.
    \end{equation}
     Because of \eqref{eq:cond-ces} and not all $\alpha_j$'s vanishing, we additionally have
    \[
       \left(\sum_{j=1}^m \alpha_j(a_j - f_j(\widehat x))^{\rho} \right)^{\frac{\kappa}{\rho}} >0.
    \]
    Hence, we deduce
      \begin{equation}\label{eq:cond-ces1}
   \alpha_j \neq 0 \Rightarrow f_j(\widetilde x) < a_j.
\end{equation}
    Altogether, due to the strict decreasing monotonicity of the function $z^{\frac{\kappa}{\rho}}$ for $\rho < 0$, we obtain in view of \eqref{eq:cond-ces} and \eqref{eq:cond-ces1}
    \[
    \sum_{j=1}^m \alpha_j(a_j - f_j(\widetilde x))^{\rho} 
< \sum_{j=1}^m \alpha_j(a_j - f_j(\widehat x))^{\rho}.
    \]
    This provides a contradiction to the fact that $\widehat x$ solves \eqref{eq: log ces1}.

    (ii) The proof is analogous to (i).
\end{proof}

\begin{theorem}
 \label{thm:w-ces-pos}
  The following statements hold for $\rho \in (0,1]$:
  \begin{itemize}
      \item[(i)] if $\widehat{x} \in P$ is feasible for \eqref{eq:ineq}, then there exist ${\alpha}_j > 0$, $j=1,\ldots,m$, such that $\widehat{x}$ belongs to the solution set of \eqref{eq:sobjutil-ces};
      \item[(ii)] if $\widehat{x} \in W$ is Slater for \eqref{eq:ineq}, then there exist ${\alpha}_j \geq 0$, $j=1,\ldots,m$ with $\exists\overline{\jmath}\in \{1,\ldots,m\}: \alpha_{\overline{\jmath}}>0$, such that $\widehat{x}$ belongs to the solution set of \eqref{eq:sobjutil-ces}.
  \end{itemize}  
\end{theorem}
\begin{proof}
  (i)  Preliminary, consider for $j=1,\ldots,m$:
    \[
    \varphi_j(x)\triangleq \left\{\begin{array}{cl}
-(a_j-f_j(x))^{\rho} & \text{if} \; f_j(x) \leq a_j\\
+\infty & \text{otherwise}.
\end{array}\right.	 
    \] 
     We show that $\widehat x$ is Pareto optimal for 
    \begin{equation}\label{eq:mobj phi1-ces-pos}
    \begin{array}{cl}
    \underset{x}{\mbox{minimize}} & \varphi(x) \triangleq \big(\varphi_1(x), \ldots, \varphi_m(x)\big)^\mathsf{T}\\
    \mbox{s.t.} & x \in K.
    \end{array}
    \end{equation}
    Otherwise, there would exit $x \in K$, such that
    \begin{equation}\label{eq:contra-ces-pos}
        \forall j \in \{1,\ldots,m\} \, : \, \varphi_{j}(\widehat x) \geq \varphi_{j}(x), \, \text{and} \; \;\exists j_x \in \{1,\ldots,m\}: \; \varphi_{j_x}(\widehat x) > \varphi_{j_x}(x).
    \end{equation}
    In particular, the feasibility of $x$ for \eqref{eq:ineq} follows from that of $\widehat x$.
    Hence, \eqref{eq:contra-ces-pos} implies directly
\[
        \forall j \in \{1,\ldots,m\} \, : \, f_{j}(\widehat x) \geq f_{j}(x), \, \text{and} \; \;\exists j_x \in \{1,\ldots,m\}: \; f_{j_x}(\widehat x) > f_{j_x}(x),
\]
i.e. $\widehat x$ is not Pareto optimal for \eqref{eq:mobj}, a contradiction. Further, in view of increasing monotonicity and concavity of the function $z^{\rho}$ for $\rho \in (0,1]$, reasoning similarly to the proof of Proposition \ref{prop: h concave1}, $\varphi_j$ turns out to be convex for every $j=1,\ldots,m$. By applying \cite[Theorem 4.1]{ehrgott2005multicriteria} to \eqref{eq:mobj phi1-ces-pos}, there exist $\alpha_j > 0$, $j=1,\ldots,m$, such that $\widehat{x}$ belongs to the solution set of the following problem:
\begin{equation} \label{eq: log ces1-pos}
    \begin{array}{cl}
    \underset{x}{\mbox{minimize}} & \displaystyle\sum_{j=1}^m \alpha_j \varphi_j(x)\\
    \mbox{s.t.} & x \in K.\\
    \end{array}
    \end{equation}
Now, we show that $\widehat x$ belongs also to the solution set of \eqref{eq:sobjutil-ces}. In fact, if this were not the case, there would exist $\widetilde x \in K$ such that
\[
  h_\text{CES} (\widetilde x) > h_\text{CES} (\widehat x). 
\]
Again, the feasibility of $\widetilde x$ for \eqref{eq:ineq} follows from that of $\widehat x$. The latter inequality becomes then 
    \[
    \left(\displaystyle \sum_{j=1}^m \alpha_j(a_j - f_j(\widetilde x))^{\rho} \right)^{\frac{\kappa}{\rho}} > \left(\displaystyle \sum_{j=1}^m \alpha_j(a_j - f_j(\widehat x))^{\rho} \right)^{\frac{\kappa}{\rho}}.
    \]
   Due to the strict increasing monotonicity of the function $z^{\frac{\kappa}{\rho}}$ for $\rho \in (0,1]$, it follows
    \[
      \displaystyle -\sum_{j=1}^m \alpha_j(a_j - f_j(\widetilde x))^{\rho} < \displaystyle -\sum_{j=1}^m \alpha_j(a_j - f_j(\widehat x))^{\rho}.
    \]
    This provides a contradiction to the fact that $\widehat x$ solves \eqref{eq: log ces1-pos}.

     (ii) The proof is analogous to (i).
\end{proof}

Finally, we link the proposed scalarization approach  
via the CES utility function to the compromise programming, see e.g. \cite{zeleny1973compromise}. 

\begin{remark}
\label{rem:compr}
In compromise programming, one tries to solve the multi-objective optimization problem:
\begin{equation}\label{eq:cp}
\begin{array}{cl}
\underset{x}{\mbox{minimize}} & g(x) \triangleq \big(g_1(x), \ldots, g_m(x)\big)^\mathsf{T}\\
\mbox{s.t.} & x \in K,
\end{array}
\end{equation}
by minimizing the distance between $g(x)$ and some reference point $b \in \R^m$
\begin{equation}\label{eq:cp-dist}
\begin{array}{cl}    
\underset{x}{\mbox{minimize}} & \displaystyle \text{dist}(g(x),b)\\
    \mbox{s.t.} & x \in K.\\
    \end{array}
\end{equation}
Here, the reference point satisfies $b \leq I$, where the ideal point $I$ of \eqref{eq:cp} is defined as
\[
    I_j \triangleq \min_{x \in K} g_j(x), j=1, \ldots,m.
\]
Usually, the distance in \eqref{eq:cp-dist} is taken to be generated by the weighted $p$-norm
\[
    \text{dist}(g(x),b) = \displaystyle \left(\sum_{j=1}^m \alpha_j \, (g_j(x)-b_j)^{p}\right)^{\frac{1}{p}}
\]
with the parameters $\alpha_j \geq 0$, $j=1,\ldots,m$, and $p \geq 1$. We want to examine to which extent the single-objective optimization problem \eqref{eq:sobjutil-ces} with the CES utility function can be cast into the framework of the compromise programming \eqref{eq:cp-dist} with the $p$-norm. For that, we set
\[
   g_j(x) = \frac{1}{a_j-f_j(x)}, j=1, \ldots,m,
\]
where we focus on the Slater points of \eqref{eq:ineq}, i.e. fulfilling $f_j(x) < a_j$, $j = 1,\ldots,m$. Note that the functions $g_j$'s remain convex under our general assumption on the convexity of $f_j$'s. For the corresponding ideal point of \eqref{eq:cp} we have
\[
   I_j = \frac{1}{\displaystyle a_j-\min_{ j=1,\ldots,m} f_j(x)} > 0, j=1, \ldots,m.
\]
Hence, an admissible choice for the reference point $b$ is the origin. Altogether, the compromise programming with respect to the $p$-norm is
\[
\begin{array}{cl}    
\underset{x}{\mbox{minimize}} & \displaystyle\displaystyle \left(\sum_{j=1}^m \alpha_j \, \left(\frac{1}{a_j-f_j(x)}\right)^{p}\right)^{\frac{1}{p}}\\
    \mbox{s.t.} & x \in K.\\
    \end{array}
\]
By redefining the parameters with $\rho = -p$, we equivalently obtain
\[
\begin{array}{cl}    
\underset{x}{\mbox{maximize}} & \displaystyle\displaystyle \left(\sum_{j=1}^m \alpha_j \,( a_j-f_j(x))^{\rho}\right)^{\frac{1}{\rho}}\\
    \mbox{s.t.} & x \in K.\\
    \end{array}
\]
The latter coincides with \eqref{eq:sobjutil-ces} if choosing $\kappa=1$. However, we note that the compromise programming covers just the case $\rho \leq -1$, whereas our approach enlarges the parameter space to $\rho \leq 1$ with $\rho \neq 0$. In other words, the proposed scalarization via the CES utility function goes beyond the $p$-norm for measuring distance in the compromise programming.
\end{remark}

\section{Algorithmic developments}
\label{sec:algo}

In the developments of this section, we assume $f_j$, $j=1,\ldots,m$, to be continuously differentiable, and $u$ to be continuously differentiable on the interior of its domain. These conditions make $h$ continuously differentiable on the open set of Slater points for \eqref{eq:ineq}.

The following generalization of Proposition \ref{prop: h concave1} to encompass pseudoconcave utility functions is instrumental for algorithmic purposes. The analysis reveals some subtleties and conditions (i) and (ii) in Definition \ref{def:barrier} turn out to play a role.    
\begin{theorem} \label{th: h concave}
Let $u$ be a barrier, monotone on its domain and pseudoconcave on the interior of its domain. Any Slater point $x^* \in K$ for \eqref{eq:ineq} such that 
    \begin{equation} \label{eq: stationary}
    \nabla h(x^*)^\mathsf{T}(x-x^*)\leq 0, \; \forall x \in K,
    \end{equation}
    belongs to the solution set of \eqref{eq:sobjutil1}.
\end{theorem}
\begin{proof}
In view of Proposition \ref{prop: h concave1} (ii), $h$ is pseudoconcave on the set of Slater points for \eqref{eq:ineq}. Thus, for every Slater point $x \in K$ for \eqref{eq:ineq},
\[
\nabla h(x^*)^\mathsf{T}(x-x^*)\leq 0 \implies h(x^*) \geq h(x) > \overline u,
\]
where the last inequality follows from (ii) in Definition \ref{def:barrier}. Moreover, in view of (i) in Definition \ref{def:barrier}, $\overline u \geq h(z)$ for every $z \in K$ such that $a - f(z) \notin \mathbb R^m_{++}$, and the claim follows.
\end{proof}

\noindent We propose a rather general algorithmic scheme to address \eqref{eq:sobjutil1} when relying on barrier utility functions. 
\begin{algorithm}[h]
\caption{General Scheme for \eqref{eq:sobjutil1} \label{alg:one}}
\KwData{$\beta >0$, $\gamma \in (0,1)$}
{Compute $x^0 \in K$ such that $h(x^0) > \overline u$}

\For{$k = 0,\ldots$}
{
Compute a direction $d^k \neq 0$ that is feasible for $K$ at $x^k$ and such that
    \begin{equation} \label{eq: ascent dir}
        \nabla h(x^k)^\mathsf{T}d^k \geq \beta \|d^k\|^2;
    \end{equation}{\label{stepalpha}}\\
Compute a stepsize $\alpha^k>0$ such that $x^k + \alpha^kd^k \in K$ and 
    \begin{equation} \label{eq: func increase}
    {h(x^k + \alpha^k d^k)\geq h(x^k)} + \gamma\alpha^k\nabla h(x^k)^\mathsf{T}d^k;
    \end{equation}\\
Set $x^{k+1} = x^k + \alpha^kd^k$\;
}
\end{algorithm}
By means of Lemma \ref{th: algorithm results}, we state some general properties of Algorithm \ref{alg:one}.
\begin{lemma} \label{th: algorithm results}
    Let $u$ be a barrier and monotone on its domain. Assume that Slater condition is satisfied for \eqref{eq:ineq}. %Assuming \eqref{eq:ineq} to satisfy the Slater condition, $x^0 \in K$ exists such that $h(x^0) > \overline u$. %Consider the sequences generated by Algorithm \ref{alg:one}.
If feasible directions $d^k \neq 0$ for $K$ are available at $x^k$ and satisfy \eqref{eq: ascent dir} for every $k$, the following statements hold for Algorithm \ref{alg:one}:
    \begin{enumerate}[label=(\roman*)]
        \item sequence $\{\alpha^k\}$ is well defined;
        \item $x^k$ is a Slater point for \eqref{eq:ineq}, for every $k$;%$x^k \in K \cap \, \{z \in \mathbb R^n: a - f(z) \in \mathbb R^m_{++}\} \subseteq K \cap \, int \, dom \, h$ for every $k$;
        %\item $x^k \in K \cap \, int \, dom \, h$ for every $k$;
        \item $\displaystyle \lim_{k\to \infty}\alpha^k\|d^k\|^2 = 0$;
        \item whenever $h$ is pseudoconcave and $\nabla h$ is locally Lipschitz on the set of Slater points of \eqref{eq:ineq}, any $\alpha^k \in \left(0, \frac{2(1-\gamma)\beta)}{L}\right]$ satisfies condition \eqref{eq: func increase}, for every $k$. Here, $L$ denotes the Lipschitz constant of $\nabla h$ on the upper level set
        $\mathcal{L}_{x_0}  \triangleq \left\{ x \in K : h(x) \geq h(x_0)\right\}$.
    \end{enumerate}
    In particular, if the direction is computed in a projected gradient fashion, i.e.,
        \begin{equation} \label{eq: direction definition}
        d^k = Proj_{K}(x^k + \tau \nabla h(x^k) )-x^k, \quad \tau >0,
        \end{equation}
    \begin{enumerate}
        \item[(v)]
    \eqref{eq: ascent dir} holds with $\beta = 1/\tau$ and Step \ref{stepalpha} is well-defined, for every $k$.
    \end{enumerate}
\end{lemma}
\begin{proof}
(i) We show that a threshold $\overline{\alpha}^k>0$ exists such that \eqref{eq: func increase} holds for all $\alpha^k \in(0,\overline{\alpha}^k]$ and $x^{k+1} = x^k + \alpha^k d^k \in K$.
    Suppose by contradiction that a sequence $\{\alpha^k_j\}$ exists such that  $\alpha^k_j \downarrow 0$  and
    \[
     \frac{h(x^k + \alpha_j^k d^k)- h(x^k)}{\alpha_j^k} < \gamma\nabla h(x^k)^\mathsf{T}d^k.
    \]
     Taking the limit $j \rightarrow \infty$ in the relation above, we get $\nabla h(x^k)^\mathsf{T}d^k \leq \gamma\nabla h(x^k)^\mathsf{T}d^k$, which is an absurdum since $\gamma <1$ and $\nabla h(x^k)^\mathsf{T}d^k> 0$ by \eqref{eq: ascent dir}.  

    (ii) Preliminarily, we observe that, in view of (ii) in Definition \ref{def:barrier}, any Slater point for \eqref{eq:ineq} can be taken as $x^0$. Assuming that $x^k \in K$ is a Slater point for \eqref{eq:ineq}, conditions \eqref{eq: func increase} and \eqref{eq: ascent dir} yield 
    \begin{equation} \label{eq:ascent of h}
    h(x^{k+1})={h(x^k + \alpha^k d^k)\geq h(x^k)} + \gamma\alpha^k\nabla h(x^k)^\mathsf{T}d^k > h(x^k) > \overline{u}, %\geq h(x^0) > \overline{u}.
    \end{equation}
    where the last inequality follows from (ii) in Definition \ref{def:barrier}. If $x^{k+1}$ is not a Slater point for \eqref{eq:ineq}, then $h(x^{k+1}) \le \overline u$ by (i) in Definition \ref{def:barrier}, a contradiction. Overall, (ii) is true by induction.
  
(iii) In view of \eqref{eq: func increase}, $\{h(x^k)\}$ is a monotone increasing sequence. The solvability of \eqref{eq:sobjutil1} from Proposition \ref{prop:exist} provides that the sequence $\{h(x^k)\}$ is also bounded from above. Hence, there exists $\widehat{h}$ such that $\displaystyle \lim_{k \to \infty} h(x^k)=\widehat{h}$. The assertion now follows from
    \[
    0 = \lim_{k\to \infty}(h(x^k + \alpha^k d^k)-h(x^k)) \geq \lim_{k\to \infty}\gamma\alpha^k\nabla h(x^k)^\mathsf{T}d^k \geq \gamma\beta\lim_{k\to \infty} \alpha^k \|d^k\|^2,
    \]
    by applying \eqref{eq: func increase} and \eqref{eq: ascent dir}, respectively.

(iv) The upper level set $\mathcal{L}_{x_0}$ is compact analogously to the proof of Proposition \ref{prop:exist}. It is also convex since $h$ is pseudoconcave, see Proposition \ref{prop: h concave1}, and, hence, quasiconcave.
     Moreover, $\mathcal{L}_{x_0}$ is a subset of the set of Slater points for \eqref{eq:ineq}, due to the barrier property of $u$. Overall, we can apply the descent lemma on the interval $\left[x^{k}, x^{k+1} \right] \in {\mathcal{L}}_{x^0}$ and obtain
        \[
    h(x^{k+1}) \geq h(x^k) + \alpha^k\nabla h(x^k)^\mathsf{T}d^k - \frac{(\alpha^{k})^2L}{2}\|d^k\|^2 \geq h(x^k) + \gamma\alpha^k\nabla h(x^k)^\mathsf{T}d^k,
    \]
    where the second inequality is valid because 
    \[
    \frac{2(1-\gamma)\nabla h(x^k)^\mathsf{T}d^k}{L\|d^k\|^2} \overset{\eqref{eq: ascent dir}}{\geq} \frac{2(1-\gamma)\beta}{L} \geq \alpha^k.
    \]

(v) By the characteristic property of the projection operator, we have
    \[
    (x^k + \tau \nabla h(x^k) - Proj_{K}(x^k + \tau \nabla h(x^k) ))^\mathsf{T}(x - Proj_{K}(x^k + \tau \nabla h(x^k) )) \leq 0, \quad \forall \; x \in K.
    \]
    Plugging $x =x^k$,
    \[
    \tau \nabla h(x^k)^\mathsf{T}(x^k-Proj_{K}(x^k + \tau \nabla h(x^k) ))+\|x^k-Proj_{K}(x^k + \tau \nabla h(x^k) )\|^2 \leq 0,
    \]
    which in turn yields
    \[
    \nabla h(x^k)^\mathsf{T} d^k = \nabla h(x^k)^\mathsf{T}(Proj_{K}(x^k + \tau \nabla h(x^k)) - x^k) \geq \frac{1}{\tau}\|x^k-Proj_{K}(x^k + \tau \nabla h(x^k) )\|^2 = \frac{1}{\tau}\|d^k\|^2,
    \]
    and thus \eqref{eq: ascent dir} holds with $\beta = 1/\tau$ and Step \ref{stepalpha} is well-defined.
    \end{proof}

Let us briefly comment on the backtracking technique for computing stepsizes in Algorithm \ref{alg:one}.

\begin{remark}
\label{remark:backtracking}
In order to compute a stepsize $\alpha^k$ satisfying the conditions in Step 4 of Algorithm \ref{alg:one}, one can rely on the following backtracking technique.
Set $\delta \in (0,1)$ and $\alpha^k_{base} > 0$ and find the smallest integer $j^k \geq 0$ such that  
  $\alpha^k = \delta^{j^k}\alpha^k_{base}$
    satisfies condition \eqref{eq: func increase} and $x^{k+1} \in K$.
\end{remark}

In next Theorem \ref{th:convergence}, we show the convergence properties of the scheme.

\begin{theorem} \label{th:convergence}
 Let $u$ be a barrier, monotone on its domain and pseudoconcave on the interior of its domain. Assume that Slater condition is satisfied for \eqref{eq:ineq}. Let $\nabla u$ be locally Lipschitz on the interior of its domain, and $\nabla f$ be locally Lipschitz. 
 If, for every $k$, $d^k$ is chosen according to \eqref{eq: direction definition} and $\alpha^k$ is computed relying on a backtracking technique, then every limit point $\overline{x}$ of the sequence $\{x^k\}$ generated by Algorithm \ref{alg:one} is a solution of \eqref{eq:sobjutil1}.
\end{theorem}
\begin{proof}
Recalling the proof of Theorem \ref{th: h concave}, $h$ is pseudoconcave on the set of Slater points for \eqref{eq:ineq}. Moreover, the assumptions  on $\nabla u$ and $\nabla f$ yield the local Lipschitz continuity of $\nabla h$ on the set of Slater points of \eqref{eq:ineq}. Now, a backtracking technique is restarted with a guess $\alpha^k_{base} > \frac{2(1 - \gamma) \beta}{L}$ for every $k$. By recalling $x^k+d^k \in K$, which is valid due to (iv) in Lemma \ref{th: algorithm results}, there exists then $\eta \in \left(0, \frac{2(1 - \gamma) \beta}{L}\right)$ such that $\alpha^k \ge \eta$ for every $k$. In particular, this implies
    %Step 3 is well-defined because of Theorem \ref{th: alphak bounded away} and $x^k + d^k \in K$, with $K$ convex. We have
    \[
    0 = \lim_{k\to \infty}\alpha^k\|d^k\|^2 \geq \eta\lim_{k\to \infty}\|d^k\|^2 = \eta\lim_{k\to \infty}\|Proj_{K}(x^k + \tau \nabla h(x^k) )-x^k\|^2,
    \]
    where the first equality is due to (iii) in Lemma \ref{th: algorithm results}. For every limit point of $\{x^k\}$ it holds $\overline x \in {\mathcal L}_{x^0}$, which is a convex and compact subset of the set of Slater points for \eqref{eq:ineq}, and 
    \[
    \overline{x} = Proj_{K}(\overline{x} + \tau \nabla h(\overline{x})) \iff \nabla h(\overline{x})^\mathsf{T}(x-\overline{x})\leq 0, \quad \forall x \in K.
    \]
    Thus, $\overline{x}$ is stationary for \eqref{eq:sobjutil1}, and the assertion follows due to Theorem \ref{th: h concave}.
\end{proof}

Algorithm \ref{alg:one} allows to iteratively compute Slater points for \eqref{eq:ineq} disregarding the disagreement reference point constraints. 
This comes at the price of identifying a starting Slater point $x^0 \in K$ for \eqref{eq:ineq}. 
The difficulty in finding such a point is very much problem-dependent. E.g., when dealing with bi-objective portfolio-selection problems, this can be done rather easily, see the empirical analysis in \cite[Section 3]{cesarone2018risk}. In more general situations, one can rely on Algorithm \ref{alg:two}. 
The procedure there might require the computation of potentially costly projections on the constraint set \eqref{eq:ineq}. However, this is still significantly less compared to the amount of projection iterations we spare thanks to Algorithm \ref{alg:one}. 
In fact, to satisfy the condition in Step 3 of Algorithm \ref{alg:two}, just a single projected-gradient iteration suffices.
\begin{algorithm}[h]
\caption{Computation of a Slater point for \eqref{eq:ineq} \label{alg:two}}
{Set $\overline{J}=\emptyset$\;}
\For{$j = 1,\ldots$,m}
{
\uIf{$x \in F$ exists such that $f_j(x)<a_j$}{Set $x^j = x$\;}
\Else{Set $\overline{J}=\overline{J}\cup j$\;
}
}
\uIf{$m = |\overline{J}|$}{\Return $\overline J$\;}
\Else{
\Return $\displaystyle \overline{x} = \frac{1}{m-|\overline{J}|}\sum_{j\notin \overline{J}}x^j$ and $\overline{J}$.}
\end{algorithm}  
Let us study the output of Algorithm \ref{alg:two} depending on the size of $|\overline{J}|$.  
Note that Algorithm \ref{alg:two} produces the index set 
\[
\overline{J}\triangleq\{j \in \{1,\ldots,m\}: f_j(x) = a_j, \forall x \in F\},
\]
where we denote for brevity the set given by disagreement reference point constraints as
\[
   F \triangleq \left\{ x \in K : f_j(x) \leq a_j, j=1,\ldots,m\right\}.
\]

\begin{theorem}\label{th:interior starting}
   The following statements hold:
\begin{itemize}
    \item[(i)] if $|\overline{J}| = 0$, then Algorithm \ref{alg:two} returns a Slater point $\overline{x} \in K$ for \eqref{eq:ineq};
    \item[(ii)]  if $0<|\overline{J}|< m$, then $F \subseteq W$;
    \item[(iii)] if $|J| = m$, then $F \subseteq P$.
\end{itemize}
\end{theorem}
\begin{proof}
    (i) If $|\overline{J}| = 0$, then $\overline x \in K$, by its convexity. Moreover, for every $j\in\{1,\ldots,m\}$,
    \[f_j(\overline{x}) \leq \displaystyle \frac{1}{m}\sum_{i=1}^m f_j(x^i)= \displaystyle \frac{1}{m}\left[\sum_{i\neq j}f_j(x^i) + f_j(x^j)\right] < \displaystyle \frac{1}{m}m a_j = a_j,
    \]
    where the first inequality is due to the convexity of $f_j$, and the strict inequality follows from Step 3 in Algorithm \ref{alg:two}.

    (ii) If $0<|\overline{J}|< m$, then there is $j\in\overline{J}$ such that $f_j(x)=a_j$ for every $x\in F$. Thus, Slater condition does not hold for \eqref{eq:ineq}, and, due to Proposition \ref{prop: Slat Weak}, $F \subseteq W$. 

    (iii) If $|J| = m$, we consider an arbitrary $\overline{x} \in F$. For all ${x} \in F$, we have $f( \overline{x}) = f(x) = a$. For any $ x \in K \backslash F$, there exists $j_x$ such that $f_{j_x}(\overline{x}) = a_{j_x} < f_{j_x}(x)$.  
\end{proof}

Theorem \ref{th:interior starting} shows how, whenever $|\overline{J}| = 0$, one can compute Pareto optimal points in \eqref{eq:ineq} through Algorithm \ref{alg:one} with $x^0 = \overline{x}$, by using a suitable strictly monotone utility function, see Theorem \ref{th:convergence}.
Whenever $|\overline{J}| = m$, any point in \eqref{eq:ineq} is a Pareto optimal point for \eqref{eq:mobj}.
Lastly, if $0<|\overline{J}|<m$, all points in \eqref{eq:ineq} are weak Pareto optimal points for \eqref{eq:mobj}. In this case, Algorithm \ref{alg:one} cannot be implemented since the Slater condition for \eqref{eq:ineq} is violated and it is impossible to find a starting point.
In order to nevertheless compute Pareto optimal points in \eqref{eq:ineq} also for this case, we introduce the ``reduced'' problem
\begin{equation} \label{eq:mobj reduced}
    \begin{array}{cl}
    \underset{x}{\mbox{minimize}} & \widetilde f(x) \triangleq \big[f_j(x)\big]_{j \in  \{ 1,\ldots,m\} \setminus \overline{J}}^\top\\
    \mbox{s.t.} & x \in \widetilde K \triangleq \left\{ y \in K : f_j(y) \leq a_j, j\in \overline{J}\right\}.
    \end{array}
\end{equation}
%Even if $\overline J \neq \emptyset$, and thus $\left\{ y \in K : f_j(y) \leq a_j, j=1,\ldots,m\right\} \subseteq W$ (see Theorem \ref{th:interior starting}), one can 
    Algorithm \ref{alg:two}, when applied to \eqref{eq:mobj reduced}, returns  a Slater point for the set of corresponding disagreement reference point constraints 
    \[
       \widetilde F \triangleq \{x \in \widetilde K : f_j(x) \leq a_j, j \in \{1,\ldots,m\}\backslash \overline{J}\}.
    \]
    Afterwards, it can be used as a starting point for Algorithm \ref{alg:one}, now applied to the ``reduced'' problem \eqref{eq:mobj reduced}. Its Pareto optimal points can be thus computed by using a suitable strictly monotone utility function, see Theorem \ref{th:convergence}.
  Thanks to Proposition \ref{prop:problequiv}, whose proof can be related to the classical results on the $\varepsilon$-constraints approach, the computed Pareto optimal points for \eqref{eq:mobj reduced} turn out to be also Pareto optimal points for \eqref{eq:mobj}.

\begin{proposition}\label{prop:problequiv}
 $\widehat x \in P \cap F$ if and only if $\widehat x \in \widetilde F$ is Pareto optimal for \eqref{eq:mobj reduced}.
\end{proposition}
\begin{proof}
Preliminarily, we note that $F=\widetilde F$. 
First, we show that if $\widehat x \in P\cap F$, then it is also Pareto optimal for \eqref{eq:mobj reduced}. If not, there would exist $x \in \widetilde K$ such that
$$
\;\forall j \in \{1,\ldots,m\}\backslash \overline{J}: \; f_{j}(\widehat x) \geq f_{j}(x), \, \text{and} \; \;
  \exists j_x \in \{1,\ldots,m\}\backslash \overline{J} \, : \, f_{j_x}(\widehat x) > f_{j_x}(x).
$$
In particular, $x \in K$ and $f_j(x) \leq a_j$ for all $j \in \overline{J}$.
Due to the construction of $\overline{J}$ and $\widehat x\in F$, we also have $f_j(\widehat x) = a_j$ for all $j \in \overline{J}$. Altogether, it holds for $x \in K$:
$$
\;\forall j \in \{1,\ldots,m\}: \; f_{j}(\widehat x) \geq f_{j}(x), \, \text{and} \; \;
  \exists j_x \in \{1,\ldots,m\} \, : \, f_{j_x}(\widehat x) > f_{j_x}(x),
$$
i.e. $\widehat x \not \in P$, a contradiction. 

For the converse, we start with  $\widehat x \in \widetilde F$ being Pareto optimal for \eqref{eq:mobj reduced}. If $\widehat x \not \in P$, 
there would exist $x \in K$ such that
$$
\;\forall j \in \{1,\ldots,m\}: \; f_{j}(\widehat x) \geq f_{j}(x), \, \text{and} \; \;
  \exists j_x \in \{1,\ldots,m\} \, : \, f_{j_x}(\widehat x) > f_{j_x}(x).
$$
Assume that $j_x \in \overline{J}$. Due to the construction of $\overline{J}$ and $\widehat x \in F$, we have $f_{j_x}(\widehat x) = a_{j_x}$. Since $f_{j}(x) \leq f_{j}(\widehat x)$ and $f_{j}(\widehat x) \leq a_j$ for all $j \in \{1,\ldots,m\}$, we obtain $x \in F$. By the same reasoning, $f_{j_x}(x) = a_{j_x}$ follows. Then, $f_{j_x}(\widehat x)=f_{j_x}(x)$, a contradiction, and, hence, $j_x \not \in \overline{J}$.  Altogether, it holds for $x \in \widetilde K$:
$$
\;\forall j \in \{1,\ldots,m\}\backslash \overline{J}: \; f_{j}(\widehat x) \geq f_{j}(x), \, \text{and} \; \;
  \exists j_x \in \{1,\ldots,m\}\backslash \overline{J} \, : \, f_{j_x}(\widehat x) > f_{j_x}(x),
$$
i.e. $\widehat x$ is not Pareto optimal for \eqref{eq:mobj reduced}, a contradiction.
\end{proof}
Summarizing, the combination of Algorithms \ref{alg:one} and \ref{alg:two} allows one to compute Pareto optimal points for \eqref{eq:mobj} even in the case where the Slater condition for \eqref{eq:ineq} is violated.

\section{Portfolio-selection and numerical results}\label{Sec:numerical}

We tackle a portfolio-selection problem where an investor is considering $n$ assets of a market and chooses the fractions $x \in \mathbb R^n$ of their budget to invest in each one, while optimizing multiple objectives.
The classical criteria we consider are the portfolio's risk (to be minimized),
\[
f_1(x) \triangleq \frac{1}{2}x^\top \Sigma x,
\]
 where $\Sigma \in \mathbb R^{n\times n}$ is the  positive semidefinite covariance matrix of assets' returns, and the portfolio's expected return (to be maximized),
\[
f_2(x) = \mu^\top x,
\]
 where $\mu \in \mathbb R^n$ are the assets' expected return.
Thirdly, we consider the sustainability-oriented criterion given by the portfolio's Environmental, Social, and corporate Governance (ESG) score (to be maximized),
\[
f_3(x) = ESG^\top x,
\]
where $ESG \in \mathbb R^n$ are the assets' ESG scores, see \cite{cesarone2023bilevel,lampariello2023solving} for a more in-depth description of the ESG-related term.
The account owners can invest between $0  \%$ and $100\%$ of their budget in each asset, and must invest the whole budget. The feasible set of the considered multi-objective problem is the following: 
\[
K=\left\{x \in \mathbb R^n: \sum_{i=1}^n x_i = 1, 0 \leq x \leq 1\right\}.\]
Notice that all three objectives are convex functions, and the feasible set is convex.

We consider a dataset consisting in daily prices, adjusted for dividends and stock splits, daily traded volumes and daily ESG scores (from 01/01/2019 to 31/12/2020) downloaded from Thomson Reuters Datastream. Specifically, we consider the Dow Jones Industrial Average (DJIA), composed of $n=28$ assets and the NASDAQ 100 (NDX), composed of $n=91$ assets.

 We adopt the equally weighted portfolio $\overline{x} = \left(\frac{1}{n},\ldots,\frac{1}{n}\right)$ as the reference. Therefore, we get the reference point $a = f(\overline{x}) = [1.3737e-04, 7.3730e-04, 7.8139e+01]$. In order to compute the Slater starting point, we use Algorithm \ref{alg:two}. To satisfy the condition $f_j(x) < a_j$ there, we address the optimization problem 
 \[
    \begin{array}{cl}
    \underset{x}{\mbox{minimize}} & f_j(x)\\
    \mbox{s.t.} & x \in \left\{ y \in K : f_j(y) \leq a_j, j=1,2,3\right\}
    \end{array}
\]
for $j = 1,2,3$, relying on just $15$ iterations of the MATLAB built-in function \texttt{quadprog}. For the specific case we consider, it holds $|\overline{J}| = 0$, and Algorithm \ref{alg:two} computes quite efficiently a Slater point to be used as $x^0$ for Algorithm \ref{alg:one}, see Theorem \ref{th:interior starting}.

 We test Algorithm \ref{alg:one} with two different utility functions: $u_\text{CD}$ from \eqref{eq:CobbDoug} with $\alpha_j = \frac{1}{3}$ for $j = 1,2,3$, and $u_\text{CES}$ from \eqref{eq:CES} with $\alpha_j = 1$ for $j = 1,2,3$, $\kappa = 1$, and $\rho = -1/2$. We use a projected-gradient method, where the stepsizes $\alpha^k$, that satisfy the conditions in Step 3 with $\gamma = 0.5$, are obtained by the backtracking procedure with $\delta = 0.5$ and $\alpha_{base}^k = 1$ for $u_{\text{CD}}$, and $\alpha_{base}^k = 50$ for $u_{\text{CES}}$, for every $k$, see Remark \ref{remark:backtracking}. 
 In our setting, the projection on $K$ is quite efficient, as we can use a closed-form solution from \cite{lampariello2021equilibrium}, and do not need to solve an optimization problem at each iteration.

To visualize an approximation of the efficient frontier, we use the constraint $f_3(x) = \overline f_3$ for $100$ different equally spaced values of $\overline f_3$, from $\displaystyle \min_{x \in K} f_3(x)$ to $\displaystyle \max_{x \in K} f_3(x)$. For each of them we additionally solve $\displaystyle\min_{x \in K} \lambda f_1(x) + (1-\lambda) f_2(x)$ for $1000$ equally spaced values of $\lambda$, from $0$ to $1$, resulting in $100000$ Pareto-efficient points.
Figures \ref{fig: efficient frontier} and \ref{fig: efficient frontier 2} show the points of the efficient frontier (blue) that dominate the reference level $a$ (cyan), as well as the starting point $x^0$ (green) and the optimal points for $u_{\text{CD}}$ (violet) and for $u_{\text{CES}}$ (red) in the objective space for the DIJA and the NDX dataset, respectively. It is clear that for both scalarizations considered, our method leads to Pareto optimal points. In Table \ref{tab: values}, we show the values of the three objectives and of the utility function when evaluated at the starting point $x^0$ and at the final one $x^*$ for the two datasets, and for the two utility functions considered. 

 \begin{figure}[h] 
\caption{Efficient frontier (in blue), reference level $a$ (cyan), starting point $x^0$ (green), optimal points for $u_{\text{CD}}$ (violet) and for $u_{\text{CES}}$ (red) in the objective space for the DIJA dataset}
\centering
\includegraphics[width=0.8\textwidth]{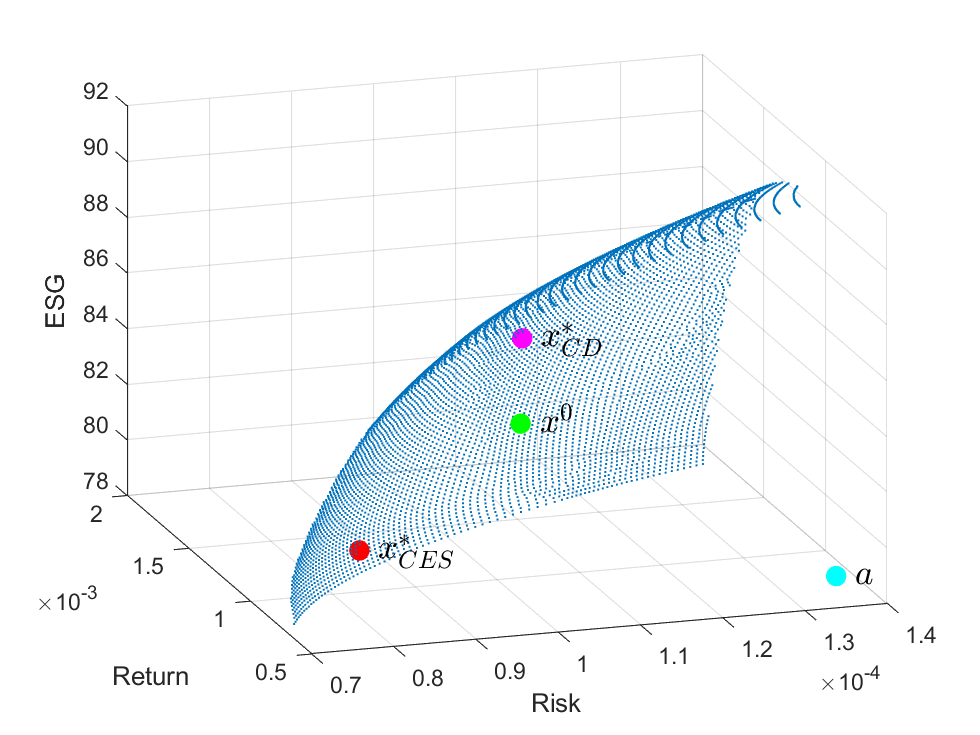}\label{fig: efficient frontier}
\end{figure}

\begin{figure}[h] 
\caption{Efficient frontier (in blue), reference level $a$ (cyan), starting point $x^0$ (green), optimal points for $u_{\text{CD}}$ (violet) and for $u_{\text{CES}}$ (red) in the objective space for the  NDX dataset}
\centering
\includegraphics[width=0.8\textwidth]{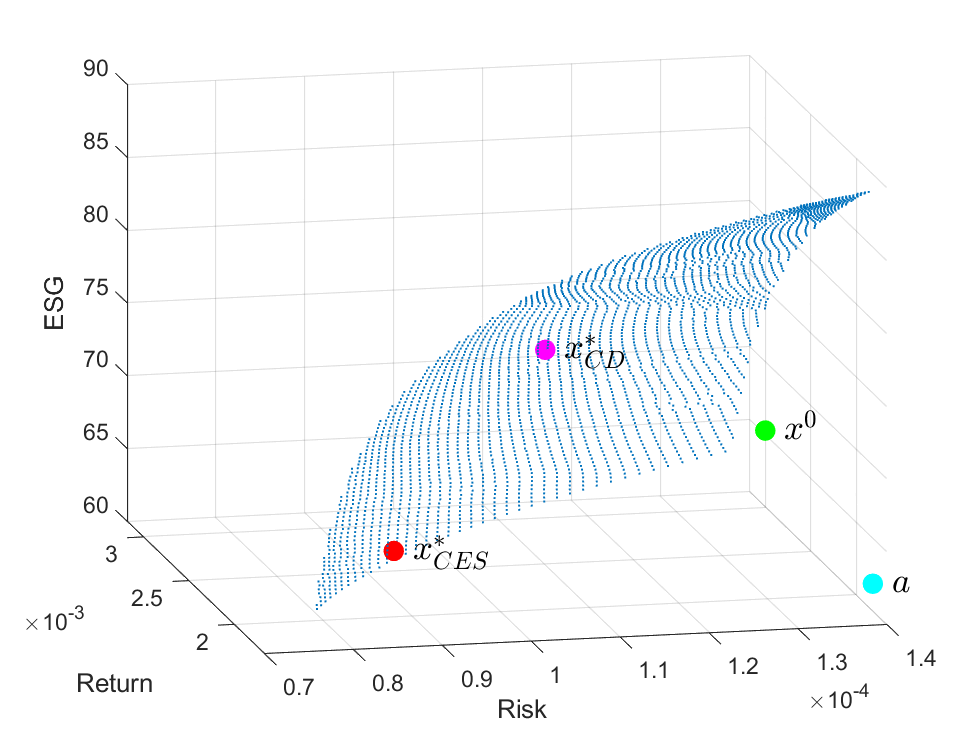} \label{fig: efficient frontier 2}
\end{figure}

\begin{table}[]
\resizebox{\textwidth}{!}{%
\begin{tabular}{cccccccccc}
\multicolumn{1}{l}{}  &                  & $f_1(x^0)$                  & $f_2(x^0)$                  & $f_3(x^0)$                  & $h(x^0)$   & $f_1(x^*)$ & $f_2(x^*)$ & $f_3(x^*)$ & $h(x^*)$   \\ \hline
\multirow{2}{*}{DIJA} & $u_{\text{CD}}$  & \multirow{2}{*}{1.0503e-04} & \multirow{2}{*}{1.1416e-03} & \multirow{2}{*}{8.2926e+01} & 3.9702e-03 & 1.0618e-04 & 1.2060e-03 & 8.5719e+01 & 4.8025e-03 \\
                      & $u_{\text{CES}}$ &                             &                             &                             & 1.9571e-05 & 8.4997e-05 & 1.1139e-03 & 7.8999e+01 & 2.7471e-05 \\
\multirow{2}{*}{NDX}  & $u_{\text{CD}}$  & \multirow{2}{*}{1.2827e-04} & \multirow{2}{*}{1.8537e-03} & \multirow{2}{*}{7.2507e+01} & 2.6612e-03 & 1.0777e-04 & 2.2702e-03 & 7.6142e+01 & 6.2829e-03 \\
                      & $u_{\text{CES}}$ &                             &                             &                             & 6.6940e-06 & 8.9570e-05 & 2.1294e-03 & 6.3713e+01 & 2.7590e-05
\end{tabular}}
\caption{Values of the three objectives and of the utility function when evaluated at the starting point $x^0$ and at the final one $x^*$ for the two datasets and for the two utility functions considered}
\label{tab: values}
\end{table}

%\section{Conclusions}

\section*{Appendix}

\begin{proof-app}{prop: cobb conc1}
 (i) 
Let us consider the auxiliary functions
   \[
      g(y) \triangleq \prod_{j=1}^m  \, y_j^{\frac{\alpha_j}{\alpha}}, y > 0, \quad w(t) \triangleq t^{\alpha}, t > 0.
   \]
   It is straightforward to check that $g$ is log-concave. Since the logarithm is monotonically increasing, the convexity of the upper level sets of $g$ follows. This means that $g$ is quasiconcave. Moreover, the function $g$ is homogeneous of degree one. Any quasiconcave function, which is additionally homogeneous of degree one, is concave, see \cite{silberberg2000structure}. Overall, we have just shown that $g$ is concave. Taking into account $\alpha \in (0,1]$ by assumption, $w$ is monotonically increasing and concave. Since the composition of a monotonically increasing and concave function with a concave function preserves concavity, $u_\text{CD} = w\circ g$ is concave on the interior of its domain. Since $u_\text{CD}$ is continuous on its domain and $\text{dom } u_\text{CD}$ is a convex set, the assertion follows.

   (ii) 
   For $\widehat y, \widetilde y \in \text{int dom }u_\text{CD}$ with $\widehat y \neq \widetilde y$, and $t \in (0,1)$, we have 
     \begin{equation}
         \label{eq:conv-cd-h}
         \begin{array}{rcl}
u_\text{CD}(t \widehat y +(1-t) \widetilde y) & = & w (g(t \widehat y +(1-t) \widetilde y)) \geq w (tg(\widehat y) +(1-t) g(\widetilde y))\\ \\
& \geq & t w (g(\widehat y)) +(1-t) w(g(\widetilde y)) = t u_\text{CES}(\widehat y) +(1-t) u_\text{CD}(\widetilde y).
 \end{array}
     \end{equation}
Here, the first inequality is due to the concavity of $g$, which is shown in (i), and the increasing monotonicity of $w$. The last inequality follows from the concavity of $w$. Moreover, $w$ is strictly monotonically increasing and strictly concave by the choice of $\alpha < 1$. Hence, we obtain at least one strict inequality in \eqref{eq:conv-cd-h} if $g(t \widehat y +(1-t) \widetilde y) > tg(\widehat y) +(1-t) g(\widetilde y)$ or $g(\widehat y) \neq g(\widetilde y)$. Then, the strict concavity of $u_\text{CD}$ on the interior of its domain would be shown. Assume on the contrary that 
 \begin{equation}
         \label{eq:conv-cd-h1}
g(t \widehat y +(1-t) \widetilde y) = tg(\widehat y) +(1-t) g(\widetilde y) \mbox{ and } g(\widehat y) = g(\widetilde y).
\end{equation}
We define the index subset
\[
   J_{\neq}\triangleq \{j \in \{1,\ldots,m \}: \widehat y_j \neq \widetilde y_j\}\neq \emptyset.
\]
It follows from \eqref{eq:conv-cd-h1}
\[
\prod_{j \in J_{\neq}}  \, (t \widehat y_j +(1-t) \widetilde y_j)^{\frac{\alpha_j}{\alpha}}=
\prod_{j \in J_{\neq}}  \, \widehat y_j^{\frac{\alpha_j}{\alpha}}=\prod_{j \in J_{\neq}}  \,  \widetilde y_j^{\frac{\alpha_j}{\alpha}},
\]
or, equivalently, by taking the logarithm
\begin{equation}
         \label{eq:conv-cd-h2}
  \sum_{j \in J_{\neq}}  \, \frac{\alpha_j}{\alpha}\log (t \widehat y_j +(1-t) \widetilde y_j)=
\sum_{j \in J_{\neq}}  \, \frac{\alpha_j}{\alpha}\log \widehat y_j=\sum_{j \in J_{\neq}}  \,  \frac{\alpha_j}{\alpha} \log \widetilde y_j.
\end{equation}
But, the strict concavity of the logarithm provides
\[
    \log (t \widehat y_j +(1-t) \widetilde y_j) >
    t \log \widehat y_j +(1-t) \log \widetilde y_j, \quad \forall j \in J_{\neq}.
\]
By multiplying these inequalities with $\frac{\alpha_j}{\alpha}> 0$, $j \in J_{\neq}$, and summing up, we obtain a contradiction together with \eqref{eq:conv-cd-h2}.
   
   (iii) Clearly, $u_\text{CD}$ is log-concave on the interior of its domain. In particular, the function $\log u_\text{CD}$ is pseudoconcave. Since the logarithm is monotonically increasing and $\nabla \log u_\text{CD}(y) = \frac{1}{u_\text{CD}(y)} \nabla u_\text{CD}(y)$, we easily deduce that $u_\text{CD}$ is also pseudoconcave.   
\end{proof-app}

\begin{proof-app}{prop: ces conc1}
%    After multiplication by a constant we may achieve that the input weights add up to one, i.e. without loss of generality
%    \[
%       \sum_{j=1}^{m} \alpha_j =1.
%    \]
%    Then,  can be applied.
(i)  
   Let us consider the auxiliary functions
   \[
      g(y) \triangleq \sum_{j=1}^m \alpha_j \, y_j^{\rho}, y >0, \quad v(z) \triangleq z^{\frac{1}{\rho}}, z >0, \quad w(t) \triangleq t^\kappa, t>0.
   \]
   For $\rho \in (0,1]$, $g$ is concave and $v$ is monotonically increasing. For $\rho \in (-\infty,0)$, $g$ is convex and $v$ is monotonically decreasing. In both cases, the convexity of the upper level sets of $v \circ g$ follows. This means that $v \circ g$ is quasiconcave. Moreover, the function $v \circ g$ is homogeneous of degree one. Any quasiconcave function, which is additionally homogeneous of degree one, is concave, see \cite{silberberg2000structure}. Overall, we have just shown that $v \circ g$ is concave. Taking into account $\kappa \in (0,1]$, $w$ is monotonically increasing and concave. Since the composition of a monotonically increasing and concave function with a concave function preserves concavity, $u_\text{CES} = w\circ (v \circ g)$ is concave on the interior of its domain. Since $u_\text{CES}$ is continuous on its domain and $\text{dom } u_\text{CES}$ is a convex set, the assertion follows.

   %The assertion mainly follows from  \cite[Theorem 3.1]{}. It has been shown there that $u_\text{CES}$ is strictly concave on the interior of its domain for the case $\sum_{j}^m\alpha_j = 1$. The latter can be achieved by multiplying $u_\text{CES}$ by the positive constant $\left(\sum_{j=1}^m\alpha_j\right)^{-\frac{\kappa}{\rho}}$.

     (ii) Let us consider the auxiliary functions
   \[
      g(y) \triangleq \left(\sum_{j=1}^m \alpha_j \, y_j^{\rho}\right)^{\frac{1}{\rho}}, y \geq 0, \quad w(t) \triangleq t^\kappa, t\geq 0.
   \]
     For $\widehat y, \widetilde y \in \text{dom }u_\text{CES}$ with $\widehat y \neq \widetilde y$, and $t \in (0,1)$, we have 
     \begin{equation}
         \label{eq:conv-ces-h}
         \begin{array}{rcl}
u_\text{CES}(t \widehat y +(1-t) \widetilde y) & = & w (g(t \widehat y +(1-t) \widetilde y)) \geq w (tg(\widehat y) +(1-t) g(\widetilde y))\\ \\
& \geq & t w (g(\widehat y)) +(1-t) w(g(\widetilde y)) = t u_\text{CES}(\widehat y) +(1-t) u_\text{CES}(\widetilde y).
 \end{array}
     \end{equation}
Here, the first inequality is due to the concavity of $g$, which is shown in (i), and the increasing monotonicity of $w$. The last inequality follows from the concavity of $w$. Moreover, note that $w$ is strictly monotonically increasing and strictly concave. Hence, we obtain at least one strict inequality in \eqref{eq:conv-ces-h} if $g(t \widehat y +(1-t) \widetilde y) > tg(\widehat y) +(1-t) g(\widetilde y)$ or $g(\widehat y) \neq g(\widetilde y)$. Then, the strict concavity of $u_\text{CES}$ on its domain would be shown. Assume on the contrary that 
 \begin{equation}
         \label{eq:conv-ces-h1}
g(t \widehat y +(1-t) \widetilde y) = tg(\widehat y) +(1-t) g(\widetilde y) \mbox{ and } g(\widehat y) = g(\widetilde y).
\end{equation}
We define the index subset
\[
   J_{\neq}\triangleq \{j \in \{1,\ldots,m \}: \widehat y_j \neq \widetilde y_j\}\neq \emptyset.
\]
It follows from \eqref{eq:conv-ces-h1}
\begin{equation}
         \label{eq:conv-ces-h2}
   \sum_{j \in J_{\neq}} \alpha_j \, (t \widehat y_j +(1-t) \widetilde y_j)^{\rho} = \sum_{j \in J_{\neq}} \alpha_j \, \widehat y_j^{\rho}= \sum_{j \in J_{\neq}} \alpha_j \, \widetilde y_j^{\rho}.
\end{equation}
But, the strict concavity of the function $z^\rho$ with $\rho \in (0,1)$ provides
\begin{equation}
         \label{eq:conv-ces-h4}
    (t \widehat y_j +(1-t) \widetilde y_j)^{\rho} >
    t \widehat y_j^{\rho} +(1-t) \widetilde y_j^{\rho}, \quad \forall j \in J_{\neq}.
\end{equation}
By multiplying with positive $\alpha_j$, $j \in J_{\neq}$, and summing up, we obtain a contradiction together with \eqref{eq:conv-ces-h2}.

(iii) The proof in (ii) can be easily adjusted here. For that, we use in the last step the strict convexity of the function $z^\rho$ with $\rho <0$. Instead of \eqref{eq:conv-ces-h4}, it holds
\[
    (t \widehat y_j +(1-t) \widetilde y_j)^{\rho} <
    t \widehat y_j^{\rho} +(1-t) \widetilde y_j^{\rho}, \quad \forall j \in J_{\neq},
\]
where $\widehat y, \widetilde y \in \text{int dom }u_\text{CES}$. This also contradicts \eqref{eq:conv-ces-h2} as above, and $u_\text{CES}$ is thus strictly concave on the interior of its domain.
\end{proof-app}

\bibliographystyle{plain}
\bibliography{Surbib}

\end{document}